\documentclass[12pt,a4paper,pdfps]{article}
\usepackage[cp1251]{inputenc}
\usepackage[english]{babel}
\usepackage{color}
\usepackage[colorlinks,linkcolor=blue,filecolor=blue,citecolor=blue]{hyperref}
\usepackage{amsmath,amssymb,amsthm}
\usepackage{graphicx}
\usepackage{comment}
\usepackage{enumitem}%% package for customizing lists
\usepackage{subfigure,multirow}
\usepackage[left=30mm, top=20mm, right=20mm, bottom=20mm, nohead, footskip=7mm]{geometry}
\DeclareMathOperator{\Ker}{Ker}
\DeclareMathOperator{\Image}{Im}
\DeclareMathOperator{\rk}{rk}
\DeclareMathOperator{\interior}{int}

\DeclareMathOperator{\ad}{\mathrm{ad}}
\DeclareMathOperator{\Ad}{\mathrm{Ad}}

\DeclareMathOperator{\sspan}{\mathrm{span}}
\DeclareMathOperator{\codim}{\mathrm{codim}}

\DeclareMathOperator{\cchar}{\mathrm{char}}

\newcommand{\f}{\mathfrak{f}}
\newcommand{\g}{\mathfrak{g}}
\newcommand{\h}{\mathfrak{h}}
\newcommand{\kk}{\mathfrak{k}}

\newcommand{\W}{\mathcal{W}}
\newcommand{\orbit}{\mathcal{O}}
\newcommand{\R}{\mathbb{R}}

\newcommand{\Kk}{\mathbb{K}}
\newcommand{\id}{\mathrm{id}}

\newcommand{\GL}{\mathrm{GL}}

\newcommand{\const}{\mathrm{const}}

\theoremstyle{definition}
\newtheorem{definition}{Definition}
\newtheorem{remark}{Remark}
\newtheorem{example}{Example}

\theoremstyle{plain}
\newtheorem{corollary}{Corollary}

\newtheorem{theorem}{Theorem}
\newtheorem{proposition}{Proposition}

\newenvironment{enumerate*}%
  {\begin{enumerate}%
    \setlength{\itemsep}{1pt}%
    \setlength{\parskip}{1pt}}%
  {\end{enumerate}}

\title{Casimir functions of free nilpotent Lie groups\\ of steps three and four\footnote{This work is supported by the Russian Science Foundation
under grant 17-11-01387-P and performed in A.\,K.~Ailamazyan Program Systems
Institute of Russian Academy of Sciences.}}

\author{
A.\,V.~Podobryaev \\ A.\,K.~Ailamazyan Program Systems
Institute of RAS \\ \tt{alex@alex.botik.ru}
}

\begin{document}

\maketitle

\begin{abstract}
Any free nilpotent Lie algebra is determined by its rank and step.
We consider free nilpotent Lie algebras of steps 3, 4 and corresponding connected and simply connected Lie groups.
We construct Casimir functions of such groups, i.e., invariants of the coadjoint representation.
%Naturally the fist type of Casimir functions consists of linear functions that correspond to the elements of maximal degree of the Lie algebra.

%It appears that for step 3 and rank greater than 2 all other Casimir functions are linear on joint level sets of the Casimir functions of the first type.
%The exception is the free nilpotent Lie group of step 3 and rank 2 known as the Cartan group.
%In this case there is a quadratic one beyond Casimir functions of the first type.

%For free 4-step nilpotent Lie groups there are Casimir functions that are quadratic on joint level sets of linear ones.

For free 3-step nilpotent Lie groups we get a full description of coadjoint orbits.
It turns out that general coadjoint orbits are affine subspaces, and
special coadjoint orbits are affine subspaces or direct products of nonsingular quadrics.
%Special coadjoint orbits are joint level sets of the Casimir functions and additional functions that are constructed by a two-step free Carnot group of lower rank.

The knowledge of Casimir functions is useful for investigation of integration properties of dynamical systems and optimal control problems on Carnot groups.
In particular, for some wide class of time-optimal problems on 3-step free Carnot groups we conclude that extremal controls corresponding to two-dimensional coadjoint orbits have the same behavior as in time-optimal problems on the Heisenberg group or on the Engel group.



\end{abstract}

\section*{\label{sec-introduction}Introduction}

The goal of this paper is a description of Casimir functions and coadjoint orbits for some class of nilpotent Lie groups. This knowledge is important for the theory of left-invariant optimal control problems on Lie groups~\cite{argachev-sachkov}. Casimir functions are integrals of the Hamiltonian system of Pontryagin maximum principle. Moreover, coadjoint orbits are invariant under the flow of the vertical part of this Hamiltonian system.

Note that left-invariant sub-Riemannian problems~\cite{agrachev-barilary-boscain}, sub-Finsler problems~\cite{barilari-boscain-ledonne-sigalotti} on Lie groups and in general problem of finding the shortest arcs for intrinsic left-invariant metrics on Lie groups~\cite{berestovskii} can be also viewed as optimal control problems.

Nilpotent Lie groups play a key role in the theory of left-invariant optimal control problems in view of existence of nilpotent approximation~\cite{agrachev-sarychev}. Any nilpotent Lie group is a quotient of a free nilpotent Lie group. A connected and simply connected free nilpotent Lie group is determined by its rank and step. Sub-Riemannian and sub-Finsler problems on such groups were investigated only for step 2, 3 and ranks 2, 3.
For step 2 it is known that rank 2 corresponds to the Heisenberg group~\cite{berestovskii-heisenberg}, rank 3 sub-Riemannian structure was completely investigated by O.~Myasnichenko~\cite{myasnichenko}, for bigger ranks there are some partial results~\cite{rizzi-serres,montanari-morbidelli,mashtakov}. For step 3 sub-Riemannian and sub-Finsler geodesics were investigated in~\cite{sachkov-didona,ardentov-ledonne-sachkov} and for sub-Finsler structures some results were obtained by Yu.\,L.~Sachkov~\cite{sachkov-subfinsler-2-step}.

It turns out that the Hamiltonian system of Pontryagin maximum principle is integrated in trigonometric functions for step 2 and in elliptic functions for step 3 and rank 2, but for bigger ranks the normal Hamiltonian system is not Liouville integrable~\cite{bizyaev-borisov-kilin-mamaev}. For sub-Riemannian structures of step more than 3 the geodesic flow is also not Liouville integrable~\cite{lokutsievskii-sachkov}.

This paper is an extension to step 3 and rank greater than 2 of the results of the paper~\cite{sachkov-two-step} on step 2.
In addition to linear Casimir functions corresponding to elements of maximal degree of a Lie algebra there are Casimir functions
that are linear on joint level sets of the first ones.
So, general coadjoint orbits are affine subspaces.

The exceptional case is the free nilpotent Lie group of step 3 and rank 2 called the Cartan group. In this case there are two independent linear Casimir functions and a quadratic one.

Special coadjoint orbits are distinguished by functions constructed by a two-step free Carnot groups of lower rank.
These orbits are affine subspaces or direct products of nonsingular quadrics.

It turns out that for free 4-step nilpotent Lie groups there are Casimir functions that are quadratic on joint level sets of linear ones.

The structure of coadjoint orbits was investigated by A.\,S.~Vorontsov~\cite{vorontsov} for any Lie algebras with the help of some reduction
based on S.\,T.~Sadetov's method~\cite{sadetov}. Let $G$ be a Lie group.
It turns out that if the corresponding Lie algebra $\g$ contains a commutative ideal $\mathfrak{i}$
(i.e., this ideal is invariant under the adjoint representation $\Ad_{\mathfrak{i}} : G \rightarrow \GL(\mathfrak{i})$),
then the coadjoint orbits on $\g^*$ are locally trivial bundles over the orbits of the action $\Ad_{\mathfrak{i}}^* : G \rightarrow \GL(\mathfrak{i}^*)$.
But in our case this commutative ideal coincides with the center of the Lie algebra and the action on $\mathfrak{i}^*$ is trivial.
The bases of corresponding bundles are points and the fibers are the coadjoint orbits.
So, this method does not work.

As an application we consider a wide class of time-optimal problems on three-step free nilpotent Lie groups.
We conclude that extremal controls corresponding to two-dimensional coadjoint orbits have the same behavior as in time-optimal problems on the Heisenberg group or on the Engel group.

The paper has the following structure. In Section~\ref{sec-defs} we give some useful definitions and notation. Section~\ref{sec-construction-of-Casimir-functions} contains a construction of Casimir functions for free Carnot groups of any step for sufficiently big ranks in invariant form (Theorem~\ref{th-casimirs-r-s}). Theorem~\ref{th-casimirs-r-3} states that for step 3 we get a complete system of Casimir functions. A description of coadjoint orbits is in Section~\ref{sec-coadjoint-orbits}. We provide an algorithm for constructing Casimir functions in Section~\ref{sec-algorithm}, also this section is an illustrative material for previous ones.
We discuss some difficulties that appear for step $s \geqslant 4$ and give a full description of Casimir functions for free Carnot groups of step 4 (Section~\ref{sec-s4}).
Finally, we consider an application to optimal control theory in Section~\ref{sec-control}, were we investigate a behavior of extremal controls corresponding to two-dimensional coadjoint orbits.

The author would like to thank Yu.\,L.~Sachkov for discussion and comments that allowed to improve the presentation of results.
Also the author is grateful to A.\,A.~Ardentov for useful discussion on sub-Finsler geometry.

\section{\label{sec-defs}Definitions and notation}

In this section we recall some basic definitions and notation on free nilpotent Lie algebras, Lie groups and coadjoint representation (see, e.g.~\cite{bourbaki,kirillov}).

\begin{definition}
\label{def-free-Lie-algebra}
Let $\f_1$ be a vector space of dimension $r \geqslant 2$ over a field $\Kk$, $\cchar{\Kk} = 0$.
A Lie algebra $\f$ is called \emph{a free Lie algebra of rank $r$} if there exists a linear map $\varphi: \f_1 \rightarrow \f$ and for any Lie algebra $\mathfrak{l}$ and a linear map
$\psi : \f_1 \rightarrow \mathfrak{l}$ there exists a homomorphism of  Lie algebras $\pi : \f \rightarrow \mathfrak{l}$ such that $\psi = \pi \circ \varphi$.
\end{definition}

\begin{definition}
\label{def-free-nilpotent-Lie-algebra}
A Lie algebra $\g$ is called \emph{a free nilpotent Lie algebra of step $s$} if it is isomorphic to a quotient of a free Lie algebra $\f$ by the ideal generated by brackets of order greater than $s$.
\end{definition}

A free nilpotent Lie algebra of step $s$ and rank $r$ is unique and naturally graded:
$$
\g = \bigoplus\limits_{m=1}^{s}{\g_m}, \qquad [\g_i, \g_j] \subset \g_{i+j}, \qquad \g_k = 0 \quad \text{for} \quad k > s.
$$
Moreover, $\g$ is generated by $\g_1$, thus it is a Carnot algebra.

\begin{definition}
\label{def-Carnot-group}
A connected and simply connected Lie group is called \emph{a step $s$ free Carnot group} if the corresponding Lie algebra is a step $s$ free nilpotent Lie algebra.
\end{definition}

\begin{definition}
\label{def-Casimir-function}
Let $G$ be a Lie group and let $\g$ be its Lie algebra. \emph{A Casimir function} is a polynomial function on the Lie coalgebra $\g^*$ invariant under the coadjoint representation.
\end{definition}

Any element $\xi \in \g$ determines a linear function $\langle \xi, \,\cdot\, \rangle : \g^* \rightarrow \Kk$, we will denote it by the same letter $\xi$. The Poisson bracket is defined on linear functions by $\{\xi, \eta\} = [\xi, \eta]$, where $\xi, \eta \in \g$.
Let $S(\g)$ the symmetric algebra generated by $\g$, viewed as an algebra of polynomial functions on the coalgebra $\g^*$.
The Poisson bracket of arbitrary polynomial functions is defined with the help of Leibnitz identity.
It is well known that Casimir functions are central elements of the algebra $S(\g)$ with respect to the Poisson bracket.

Let $B$ be the Poisson bi-vector on $\g^*$, i.e., for any $p \in \g^*$ a skew-symmetric form $B_{p} : \g \times \g \rightarrow \Kk$ is such that for any
$f, g \in S(\g)$ the Poisson bracket is $\{f, g\}(p) = B_{p}(d_{p}f, d_{p}g)$.

\section{\label{sec-construction-of-Casimir-functions}Construction of Casimir functions}

Let now $\g$ be a free nilpotent Lie algebra of rank $r$ and step $s$ and let $G$ be the corresponding Carnot group.

\begin{remark}
\label{rem-Casimir-commutes-with-g1}
A function $f \in S(\g)$ is a Casimir function if and only if $\{\g_1, f\} = 0$. Indeed, the Poisson algebra $S(\g)$ is generated by linear functions on the coalgebra $\g^*$, but the Lie algebra $\g$ is generated by $\g_1$.
\end{remark}

\begin{remark}
\label{rem-linearCasimirs}
It is obvious that any element of $\g_s$ (an element of the maximal degree of the Lie algebra) is a linear Casimir function.
\end{remark}

Consider a map $i_{p} : S(\g) \rightarrow \Kk$ that takes a polynomial on $\g^*$ into its value at a point $p \in \g^*$.
Let $I : \g^* \rightarrow S(\g_s)$ be a linear function on $\g^*$ with values in $S(\g_s)$.
Denote $\widehat{I}(p) = i_{p} I(p)$ for $p \in \g^*$.

\begin{proposition}
\label{prop-linearCasimirs}
A function $\widehat{I} : \g^* \rightarrow \Kk$ is a Casimir function if and only if $i_{p} I(\cdot) \in \Ker{B^1_{p}}$ for all $p \in \g^*$, where
$$
B^1_{p} : \g \rightarrow \g_1^*, \qquad B^1_{p}(\xi) = B_{p}(\,\cdot\,, \xi)|_{\g_1}, \qquad \xi \in \g.
$$
\end{proposition}

\begin{proof}
First, note that
\begin{equation}
\label{eq-diff}
d_{p} \widehat{I} = i_{p} I(\cdot) + d_{p} I(p).
\end{equation}

Second, according to Remark~\ref{rem-Casimir-commutes-with-g1} consider
$$
\{\g_1, \widehat{I}\}(p) = B_{p}(\g_1, d_{p}\widehat{I}) = B_{p}(\g_1, i_{p}I(\cdot)) + B_{p}(\g_1, d_{p}I(p)),
$$
where the last term is equal to $\{\g_1, I(p)\} \subset \{\g_1, S(\g_s)\} = 0$ by Remark~\ref{rem-linearCasimirs}. So, $\{\g_1, \widehat{I}\} = 0$ if and only if for any $p \in \g^*$ we have $B_{p}(\g_1, i_{p}I(\cdot)) = 0$.
\end{proof}

\begin{proposition}
\label{prop-Kerb-is-linearCasimir}
Any element of $\Ker{B^1_{p}|_{\g_{s-1}}}$ can be realized as $i_{p}I(\cdot)$, where $I: \g^* \rightarrow S(\g_s)$ is a linear map.
\end{proposition}

\begin{proof}
For $\eta \in \g_{s-1}$ and $\xi \in \g_1$ we have $B^1_{p}(\eta)(\xi) = p([\xi, \eta])$. Since $[\g_1, \g_{s-1}] \subset \g_s$, the coefficients of $B^1_{p}|_{g_{s-1}}$ are equal to values of linear functions on $\g^*$ of the form $\g_s$ at a point $p$. It is clear that elements of $\Ker{B^1_{p}|_{\g_{s-1}}}$ can be viewed as vectors of  polynomials of variables from the space $\g_s$. These polynomials have degree $\rk{B^1_{p}|_{\g_{s-1}}}$. For details see Section~\ref{sec-algorithm}. This implies that any of these elements can be realized as $i_{p}I(\cdot)$,
where $I \in \g_{s-1} \otimes S(\g_s)$.
\end{proof}

How many such Casimir functions are there? The answer is given by Propositions~\ref{prop-exist-p-full-rank}, \ref{prop-b-is-surj}.

\begin{proposition}
\label{prop-exist-p-full-rank}
If $\dim{\g_{s-1}} \geqslant \dim{\g_1} = r$, then
there exists $p \in \g^*$ such that $\Image{B^1_p|_{g_{s-1}}} = \g_1^*$.
\end{proposition}

\begin{proof}
Take a basis $\xi_1, \dots, \xi_r \in \g_1$. The elements
$$
\eta_1 = (\ad{\xi_1})^{s-2}\xi_2, \quad \eta_2 = (\ad{\xi_2})^{s-2}\xi_3, \ \dots \ , \quad \eta_r = (\ad{\xi_r})^{s-2}\xi_1 \in \g_{s-1}
$$
are linearly independent because this is a part of the Hall basis of the free nilpotent Lie algebra (see, e.g.~\cite{serre}).
The same is true for the elements
\begin{equation}
\label{eq-basis-gs}
[\xi_1, \eta_1], \ [\xi_2, \eta_2], \ \dots \ , \ [\xi_r, \eta_r] \in \g_s.
\end{equation}
Expand this system to a basis of the space $\g_s$. Take $p \in \g^*$ such that it takes the values $1$ on elements of the type~\eqref{eq-basis-gs} and zero on other basis elements.
It is clear that $B^1_p(\eta_i)(\xi_j) = \delta_{ij}$ for $i=1,\dots,r$, where $\delta_{ij}$ is the Kronecker delta.
So, the map $B^1_p|_{\g_{s-1}}$ is surjective.
\end{proof}

\begin{proposition}
\label{prop-b-is-surj}
Assume that $\dim{\g_{s-1}} \geqslant \dim{\g_1} = r$. Then for almost all $p \in \g^*$ we have $\Image{B^1_{p}|_{\g_{s-1}}} = \g_1^*$.
\end{proposition}

\begin{proof}
Assume by contradiction that for $p$ from some open set in $\g^*$ we have $\Image{B^1_{p}|_{\g_{s-1}}} \neq \g_1^*$. This condition is polynomial in coefficients of the map
$B^1_{p}|_{\g_{s-1}}$ which are elements of the subspace $\g_s$. So, if these polynomials are zero on some open set, then they are identically zero in contradiction with Proposition~\ref{prop-exist-p-full-rank}.
\end{proof}

\begin{remark}
\label{rem-complicated-proof-on-full-rank}
Of course, the condition of full rank of a linear map is open. But the coefficients of the map $B^1_p|_{\g_{s-1}}$ are not arbitrary. The goal of Propositions~\ref{prop-exist-p-full-rank}, \ref{prop-b-is-surj} is to show that since our Lie algebra is free nilpotent, there are no linear relations between rows of the matrix $B^1_p|_{\g_{s-1}}$.
\end{remark}

\begin{proposition}
\label{prop-dim-inequality}
For free nilpotent Lie algebras of step $s$ the inequality $\dim{\g_{s-1}} \geqslant \dim{\g_1}$ is satisfied for all ranks, except a finite number of ranks.
\end{proposition}

\begin{proof}
The formula for dimensions of graded components of a free nilpotent Lie algebra of rank $r$ reads as follows (see for example~\cite{serre}):
\begin{equation}
\label{eq-grad-component-dimension}
\dim{\g_m} = \frac{1}{m} \sum\limits_{d|m}{\mu(d)r^{\frac{m}{d}}}, \qquad m \leqslant s,
\end{equation}
where $\mu$ is the M\"{o}bius function.
Then the expression $\dim{\g_{s-1}} - \dim{\g_1}$ is a polynomial of the variable $r$ with the positive leading coefficient. So, it is greater than or equal to zero for sufficiently big $r$.
\end{proof}

\begin{theorem}
\label{th-casimirs-r-s}
For free Carnot groups of step $s$ of sufficiently big rank $r$ there are\\
\emph{(1)} $\dim{\g_s}$ linear Casimir functions,\\
\emph{(2)} $(\dim{\g_{s-1}} - r)$ Casimir functions that are linear on the level sets of the first ones.
The degrees of these functions are equal to $r+1$.
\end{theorem}

\begin{proof}
Immediately follows from Remark~\ref{rem-linearCasimirs} and Propositions~\ref{prop-linearCasimirs}, \ref{prop-Kerb-is-linearCasimir}, \ref{prop-b-is-surj}, \ref{prop-dim-inequality}.
\end{proof}

\begin{theorem}
\label{th-casimirs-r-3}
For a three-step free Carnot group of rank $r \geqslant 3$ there are $(r^3 - r)/3$ linear Casimir functions and $(r^2-3r)/2$ Casimir functions that are linear on the level sets of the first ones \emph{(}the degrees of these functions are equal to $r + 1$\emph{)}. This is a complete system of Casimir functions.
\end{theorem}

\begin{proof}
The first part of the theorem follows from Theorem~\ref{th-casimirs-r-s} and formula~\eqref{eq-grad-component-dimension}. Indeed, the expression $\dim{\g_2} - r = \frac{1}{2}(r^2 - 3r)$ is nonnegative if $r \geqslant 3$.

Let us prove now that we get a complete system of Casimir functions.
Let $B^{12}_{p} = B^1_{p}|_{\g_2}$ and $B^{21}_{p} : \g_1 \rightarrow \g_2^*$ is such that for $\xi \in \g_1$ we have $B^{21}_{p}(\xi) = B_{p}(\,\cdot\,, \xi)|_{\g_{2}}$.
If $\lambda_1 + \lambda_2 + \lambda_3 \in \Ker{B_p}$ and $\lambda_1 \in \g_1$, $\lambda_2 \in \g_2$, $\lambda_3 \in \g_3$, then
$\lambda_1 \in \Ker{B^{21}_p}$.
But the form $B_{p}$ is skew-symmetric, this means that $B^{21}_{p} = -(B^{12}_{p})^*$. From Proposition~\ref{prop-b-is-surj} we know that $\Image{B^{12}_{p}} = \g_1^*$ for almost all $p$, then $\Ker{(B^{12}_{p})^*} = 0$.
So, from $\lambda_1 = 0$ and $\lambda_3 \in \g_3 \subset \Ker{B_p}$ follows $\lambda_2 \in \Ker{B^{12}_p}$. We obtain $\dim{\Ker{B_{p}}} = \dim{\Ker{B^{12}_{p}}} + \dim{\g_3}$.

It is well known, e.g.~\cite{kirillov}, that coadjoint orbits are symplectic leaves of the Poisson structure on $\g^*$. So, $\dim{\Ker{B_{p}}}$ is equal to a number of Casimir functions for general $p \in \g^*$. But we had already found this amount of independent Casimir functions.
\end{proof}

\begin{example}
\label{ex-Cartan-group}
There is a special case $r = 2$. Note that the corresponding three-step free Carnot group is called \emph{the Cartan group}. The Lie algebra is generated by $\xi_1, \xi_2$ and has the following structure:
$$
\xi_{12} = [\xi_1, \xi_2], \qquad \xi_{112} = [\xi_1, \xi_{12}], \qquad \xi_{212} = [\xi_2, \xi_{12}],
$$
$$
\g_1 = \sspan{\{\xi_1, \xi_2\}}, \qquad \g_2 = \sspan{\{\xi_{12}\}}, \qquad \g_3 = \sspan{\{\xi_{112}, \xi_{212}\}}.
$$
It is easy to see that in addition to linear Casimir functions $\xi_{112}, \xi_{212}$ we have a quadratic Casimir function
\begin{equation*}
\textstyle
{{1}\over{2}}\xi_{12}^2 + \xi_1\xi_{212} - \xi_2\xi_{112},
\end{equation*}
see~\cite{sachkov-didona}.
This is a complete system of Casimir functions since $\rk{B_p} = 3$ for almost all $p \in \g^*$.

If $\xi_{112} = \xi_{212} = 0$, then corresponding coadjoint orbits are affine planes $\xi_{12} = \const \neq 0$.

If $\xi_{112} = \xi_{212} = \xi_{12} = 0$, then corresponding coadjoint orbits are points.
\end{example}

\section{\label{sec-coadjoint-orbits}Coadjoint orbits}

In this section we give a description of coadjoint orbits for three-step free Carnot groups.
General coadjoint orbits are joint level sets of the Casimir functions given in Section~\ref{sec-construction-of-Casimir-functions}.
Here we describe additional functions that determine lower-dimensional coadjoint orbits.
Below we consider $p \in \g^*$ such that $\Image{B^{12}_{p}} \neq \g_1^*$, in other words the annihilator of this image is nontrivial:
$0 \neq (\Image{B^{12}_{p}})^{\circ} \subset \g_1$.
Let $\W \ni p$ be a joint level set of all Casimir functions.

\begin{proposition}
\label{prop-same-h-for-coadjoint-orbit}
The subspace $(\Image{B^{12}_{p}})^{\circ} \subset \g$ depends only on the set $\W \subset \g^*$ but not on an element $p \in \g^*$.
\end{proposition}

\begin{proof}
The linear map $B^{12}_{p}$ depends only on functions from the subspace $\g_3$, that are constant on the set $\W$.
\end{proof}

Denote $\h_1 = (\Image{B^{12}_{p}})^{\circ} \subset \g_1$ and
$B^{11}_{p} = B_{p}|_{\g_1}$.

\begin{proposition}
\label{prop-preimageB12}
If $\gamma \in \kk = \Ker{B^{11}_p|_{\h_1}}$, then $B^{11}_p(\gamma) \in \Image{B^{12}_p}$.
\end{proposition}

\begin{proof}
We have $B^{11}_p(\h_1, \gamma) = 0$, this means that $B^{11}_p(\gamma) \in \h_1^{\circ} = \Image{B^{12}_p}$.
\end{proof}

\begin{proposition}
\label{prop-D-is-symmetric}
Consider the subspace $[\gamma, \g_1] \subset \g_2$ and the map $D = B^{12}_p|_{[\gamma, \g_1]}$.
The map $D(\ad{\gamma}) : \g_1 \rightarrow \g_1^*$ is self-adjoint for any $p \in \W$.
\end{proposition}

\begin{proof}
Indeed, for any $\xi, \zeta \in \g_1$ we have
$$
B^{12}_p(\xi, [\gamma, \zeta]) = \{\xi, \{\gamma, \zeta\}\}(p) = \{\{\xi, \gamma\}, \zeta\}(p) + \{\gamma, \{\xi, \zeta\}\}(p) = B^{12}_p(\zeta, [\gamma, \xi]),
$$
because $\{\gamma, \{\xi, \zeta\}\}|_{\W} = 0$ since $\gamma \in \h_1 = (\Image{B^{12}_p})^{\circ}$.
\end{proof}

\begin{proposition}
\label{prop-bilinear-part}
Assume that $D$ is non-degenerate.
Then for any $\gamma \in \kk$ there exists $\eta \in (B^{12}_p)^{-1}B^{11}_p(\gamma) \subset \g_2$ such that
$\eta = 2i_pA(\, \cdot \, , p)$, where
$A : \g^* \times \g^* \rightarrow S(\g_3)$ is a symmetric bilinear form.
\end{proposition}

\begin{proof}
We can take $\eta = D^{-1}B^{11}_p(\gamma)$.
By definition of $B^{11}_p$ we have $\eta = -D^{-1}(\ad^*{\gamma})p$.
So, the bilinear form $A$ on the coalgebra $\g^*$ given by the formula: $2A(p, q)= -qD^{-1}(\ad^*{\gamma})p$ for $p, q \in \g^*$.
Let us show that this bilinear form is symmetric.
First, since our Lie algebra is free nilpotent, the map $\ad{\gamma} : \g_1 \rightarrow [\gamma, \g_1]$ is an isomorphism of vector subspaces.
Second, by Proposition~\ref{prop-D-is-symmetric} the map $D(\ad{\gamma})$ is self-adjoint, then the map $(\ad^*{\gamma})^{-1}D$ is self-adjoint.
So, the opposite map $D^{-1}(\ad^*{\gamma})$ is self-adjoint as well.

Multiplying $\gamma$ by $\det{D}$ we obtain that the coefficients of $D^{-1}$ are polynomials of variables from the space $\g_3$.
\end{proof}

\begin{proposition}
\label{prop-eta-gives-quadratic}
For any $\gamma \in \kk$ there exists $\eta \in (B^{12}_p)^{-1}B^{11}_p(\gamma) \subset \g_2$ such that
$\eta = 2i_pA(\, \cdot \, , p) + i_pL(\cdot)$, where
$A : \g^* \times \g^* \rightarrow S(\g_3)$ is a symmetric bilinear form and
$L : \g^* \rightarrow S(\g^3 \oplus \Ker{D})$ is a linear map.
\end{proposition}

\begin{proof}
Take subspace $\mathfrak{n}^{\circ} \subset \Image{B^{12}_p} \subset \g_1^*$ such that $\Image{B^{12}_p} = \Image{D} \oplus \mathfrak{n}^{\circ}$.
By Proposition~\ref{prop-preimageB12} $b = B^{11}_p(\gamma) \in \Image{B^{12}_p}$.
Let $b = b_1 + b_2$, where $b_1 \in \Image{D}$ and $b_2 \in  \mathfrak{n}^{\circ}$.
Take $\eta_1 = D^{-1}b_1$ and $\eta_2 \in (B^{12}_p)^{-1}(b_2)$.

Let us prove that components of $b_2$ are functions from the space $[(\Image{D})^\circ, \gamma] = \Ker{D}$.
Indeed, for any $\xi \in \g_1$ we have $b_2(\xi) = b(\xi) - b_1(\xi) = B^{11}_p(\xi, \gamma) - B^{12}_p(\xi, \eta_1)$.
If $\xi \in \h_1 = (\Image{B^{12}_p})^{\circ}$, then the second term equals zero
and since $\gamma \in \kk = \Ker{B^{11}|_{h_1}}$ the first term is zero as well.
If $\xi \in \mathfrak{n}$, then $b_2(\xi) = 0$.
If $\xi \in (\Image{D})^\circ$, then $B^{12}_p(\xi, \eta_1) = 0$ and $B^{11}_p(\xi, \gamma) = [\xi, \gamma](p)$.

This implies that the components of $\eta_2$ are rational functions whose nominators are polynomials from the space $S(\g_3 \oplus \Ker{D})$
and denominators are polynomials from the space $S(\g_3)$.
Multiply $\gamma$ and $\eta_2$ by the common denominator.
So, the element $\eta_2$ can be viewed as $i_pL$ for a linear map $L : \g^* \rightarrow S(\g^3 \oplus \Ker{D})$.

It remains to take a symmetric bilinear form constructed by the element $\eta_1$ as in Proposition~\ref{prop-bilinear-part}.
\end{proof}

We will associate with $\gamma \in \kk$ a quadratic function (see Proposition~\ref{prop-eta-gives-quadratic})
$$
Q_{\gamma} : \g^* \rightarrow S(\g_3 \oplus \Ker{D}), \qquad Q_{\gamma}(p) = A_{\gamma}(p, p) + L_{\gamma}(p), \qquad p \in \g^*.
$$

It is clear that any element $\gamma \in \kk$ can be considered as a linear function $I_{\gamma}: \g^* \rightarrow S(\h_2)$, where $\h_2 = [\h_1, \h_1]$.
The proof is similar to the proof of Proposition~\ref{prop-Kerb-is-linearCasimir}.
But we multiplied the equation $B^{12}_p(\eta) = B^{11}_p(\gamma)$ by a function polynomial in $\g_3$ in the proofs of
Propositions~\ref{prop-bilinear-part}, \ref{prop-eta-gives-quadratic} above.
So, now $\gamma$ can be viewed as a linear function $I_{\gamma}: \g^* \rightarrow S(\h_2 \oplus \g_3)$.

Denote by $\widehat{I}_{\gamma}(p) = i_pI_{\gamma}(p)$ and $\widehat{Q}_{\gamma}(p) = i_pQ_{\gamma}(p)$ the results of substitution of $p \in \g^*$ into $I_{\gamma}(p)$ and $Q_{\gamma}(p)$, respectively.

We need Propositions~\ref{prop-h2-are-constant-on-O}, \ref{prop-KerD} below
to show that coefficients of the functions $\widehat{I}_{\gamma}$ and $\widehat{Q}_{\gamma}$
are constant on the coadjoint orbits from the set $\W$.

\begin{proposition}
\label{prop-h2-are-constant-on-O}
The equality $\{\g, \h_2\}|_{\W} = 0$ is satisfied.
\end{proposition}

\begin{proof}
It is sufficient to show that $\{\g_1, \h_2\}|_{\W} = 0$. For any $\xi \in \g_1$ and any $\gamma_1, \gamma_2 \in \h_1$ we have
$$
\{\xi, \{\gamma_1, \gamma_2\}\} = \{\{\xi, \gamma_1\}, \gamma_2\} + \{\gamma_1, \{\xi, \gamma_2\}\}.
$$
But $\{\xi, \gamma_1\} \in \g_2$ and $\{\{\xi, \gamma_1\}, \gamma_1\}|_{\W} = 0$ by definition of the subspace $\h_1$ and Proposition~\ref{prop-same-h-for-coadjoint-orbit}. The second term of the sum above is also zero by the same reason.
\end{proof}

\begin{proposition}
\label{prop-KerD}
The equality $\{\g, \Ker{D}\}|_{\W} = 0$ is satisfied.
\end{proposition}

\begin{proof}
It is sufficient to show that $\{\g_1, \Ker{D}\}|_{\W} = 0$.
This is true by definition of $\Ker{D}$.
\end{proof}

\begin{theorem}
\label{th-coadjoint-orbits}
Let $G$ be a three-step free nilpotent Carnot group of rank $r \geqslant 3$.
If $p \in \g^*$ is such that $\h_1 = (\Image{B^{12}_{p}})^{\circ} \neq 0$, then the coadjoint orbit $(\Ad^*{G})p$ is a joint level set of\\
\emph{(1)} linear Casimir functions from $\g_3$\emph{;}\\
\emph{(2)} Casimir functions constructed from $\Ker{B^{12}_{p}}$ by Proposition~\emph{\ref{prop-Kerb-is-linearCasimir}}\emph{;} \\
\emph{(3)} functions of the form $\widehat{I}_{\gamma} - \widehat{Q}_{\gamma}$ for $\gamma \in \kk = \Ker{B^{11}_{p}|_{\h_1}}$.
\end{theorem}

\begin{proof}
We proved in Section~\ref{sec-construction-of-Casimir-functions} that the functions of the first and the second types are constant on the coadjoint orbit $\orbit = (\Ad^*{G})p$.
Now we need to check that $(\widehat{I}_{\gamma} - \widehat{Q}_{\gamma})|_{\orbit} = \const$.

Consider the Hamiltonian vector field $\vec{\xi}$ on $\g^*$ corresponding to a linear function $\xi \in \g$. It is well known that this vector field coincides with the velocity field of the coadjoint action of the Lie group $G$ on the coalgebra $\g^*$. So, a function $f \in C^{\infty}(\g^*)$ is constant on the orbit $\orbit$ if and only if
for all $\xi \in \g$ we have $\vec{\xi} f = \{\xi, f\} = 0$ at any point of this orbit.
The Lie algebra $\g$ is generated by the subspace $\g_1$. So, it is sufficient to check the condition above for $\xi \in \g_1$.

Note, that by Leibnitz identity the differential of the function $\widehat{Q}_{\gamma}$ is equal to
\begin{equation}
\label{eq-diff-quad}
d_p \widehat{Q}_{\gamma} = 2i_pA_{\gamma}(\, \cdot \, , p) + i_pL_{\gamma}(\cdot) + d_pQ_{\gamma}(p).
\end{equation}

Using formulas~\eqref{eq-diff} and \eqref{eq-diff-quad} we obtain
$$
\{\xi, \widehat{I}_{\gamma}\}(p) = B_p(\xi, i_{p}I_{\gamma}(\cdot) + d_pI_{\gamma}(p)) = B^{11}_p(\xi, \gamma) + \{\xi, I_{\gamma}(p)\}(p),
$$
$$
\{\xi, \widehat{Q}_{\gamma}\}(p) = B_p(\xi, 2i_{p}A_{\gamma}(\, \cdot \, , p) + i_pL_{\gamma}(\cdot) + d_pQ_{\gamma}(p)) = B^{12}_p(\xi, \eta) + \{\xi, Q_{\gamma}(p)\}(p).
$$
We know that $I_{\gamma}(p) \in S(\h_2 \oplus \g_3)$ and $Q_{\gamma}(p) \in S(\g_3 \oplus \Ker{D})$. By Propositions~\ref{prop-h2-are-constant-on-O}, \ref{prop-KerD} we get $\{\xi, I_{\gamma}(p)\} = 0$ and $\{\xi, Q_{\gamma}(p)\} = 0$ on the orbit $\orbit$.

Consequently, $\{\xi, \widehat{I}_{\gamma} - \widehat{Q}_{\gamma}\}(p) = B^{11}_p(\xi, \gamma) - B^{12}_p(\xi, \eta) = 0$ since $\eta \in (B^{12}_p)^{-1}B^{11}_p(\gamma)$.

It remains to show that we found a complete system of functions. Since $\codim{\orbit} = \dim{\Ker{B_p}}$ let us find $\dim{\Ker{B_p}}$.
Let $\lambda \in \Ker{B_p}$ and $\lambda = \lambda_1 + \lambda_2 + \lambda_3$, where $\lambda_i \in \g_i$ for $i=1,2,3$.
It is clear that $\lambda_3 \in \Ker{B_p}$ and $\lambda_1 \in \Ker{B^{21}_p}$. But $B^{21}_p = -(B^{12}_p)^*$, then $\lambda_1 \in (\Image{B^{12}_p})^{\circ} = \h_1$.
Next, $B^{11}_p(\lambda_1) + B^{12}_p(\lambda_2) = 0$, consequently $B^{11}_p(\lambda_1) \in \Image{B^{12}_p} = \h_1^{\circ}$.
So, $\lambda_1 \in \Ker{B^{11}_p|_{\h_1}}$ and
$\lambda_2 \in (B^{12}_p)^{-1}B^{11}_p(\lambda_1)$. We get
\begin{equation}
\label{eq-codim-O}
\codim{\orbit} = \dim{\Ker{B_p}} = \dim{\Ker{B^{11}_p|_{\h_1}}} + \dim{\Ker{B^{12}_p}} + \dim{\g_3},
\end{equation}
this number coincides with the number of independent functions of types (1)--(3).
\end{proof}

\begin{remark}
\label{rem-dim-of-coadjoint-orbits}
Let $\dim{\h_1} = k_1$ and $\dim{\Ker{B^{11}_{p}|_{h_1}}} = k_2$. Then $\dim{\Ker{B^{12}_p}} = \dim{\g_2} - r + k_1$. Using formula~\eqref{eq-codim-O} we obtain
$\dim{\orbit} = 2r - (k_1 + k_2)$. Since $k_2$ is dimension of the kernel of skew-symmetric matrix of size $k_1$, these numbers have the same parity and the dimension of a coadjoint orbit is even.
\end{remark}

\begin{remark}
Any element $\gamma \in \Ker{B^{11}_p|_{\h_1}}$ determines an element $\eta \in (B^{12}_p)^{-1}B^{11}_p(\gamma)$ up to the element of $\Ker{B^{12}_p}$.
The corresponding function $\widehat{I}_{\gamma} - \widehat{Q}_{\gamma}$ is defined up to addition of a Casimir function.
\end{remark}

\begin{corollary}
\label{crl-coadjointorbitstypes}
Any coadjoint orbit of a three-step free nilpotent Carnot group is an affine subspace or a direct product of nonsingular quadrics.
If rank of the group is greater than \emph{2}, then any general coadjoint orbit is an affine subspace.
\end{corollary}

\begin{proof}
For the special case of the Cartan group ($r = 2$) see Example~\ref{ex-2358} and~\cite{sachkov-didona}. Consider the case $r \geqslant 3$.

If we fix values of the linear Casimir functions from the space $\g_3$, then the Casimir functions constructed from the space $\Ker{B^{12}_p}$ with the help of Proposition~\ref{prop-Kerb-is-linearCasimir} become linear on the joint level set of the first ones.
So, general coadjoint orbits coincide with the sets $\W$ that are affine subspaces.

By Propositions~\ref{prop-h2-are-constant-on-O}, \ref{prop-KerD} the linear functions of the forms $\h_2$ and $\Ker{D}$ are constant on coadjoint orbits from the set
$\W$. Let us fix their values. Then the functions $\widehat{I}_{\gamma}$ and $\widehat{Q}_{\gamma}$ become linear and quadratic, respectively.

Let us look at quadratic functions $\widehat{I}_{\gamma_i} - \widehat{Q}_{\gamma_i}$,
where $\gamma_i$ are basis vectors of the subspace $\Ker{B^{11}|_{\h_1}}$
First, the linear parts of these functions are linearly independent since $\gamma_i$ are linearly independent.
Second, their homogeneous quadratic parts depend on the variables from the subspaces $S_i = [\gamma_i, \g]$.
An intersection of any pair of these subsets is trivial on $\W$.
Really, let $\xi_i, \xi_j \in \g$ be such that $[\gamma_i, \xi_i] = [\gamma_j, \xi_j]$.
But our Lie algebra is free nilpotent, so, $\gamma_j = \xi_i$ and $\gamma_i = -\xi_j$, and we obtain $[\gamma_i, \xi_i] = [\gamma_i, \gamma_j]$.
Since $\gamma_i \in \Ker{B^{11}_p|_{\h_1}}$ we have $[\gamma_i, \gamma_j]|_{\W} = 0$.

Finally, it follows that the quadratic functions $\widehat{I}_{\gamma_i} - \widehat{Q}_{\gamma_i}$ depend on different sets of variables.
So, the corresponding joint level sets are direct products of quadrics.
\end{proof}

\begin{remark}
\label{rem-h-has-step-2-and-lower-rank}
The additional functions that determine coadjoint orbits are constructed by a free Carnot group of step 2 and lower rank.
Really, the subspace $\h_1|_{\W}$ (elements of $\h_1$ regarded as linear functions on the set $\W$) generates a two-step free nilpotent Lie algebra because
$\{\h_1, \h_2\}|_{\W} = 0$ by Proposition~\ref{prop-h2-are-constant-on-O}.
\end{remark}

\section{\label{sec-algorithm}Algorithm}

In this section we give an algorithm for constructing Casimir functions based on results of Sections~\ref{sec-construction-of-Casimir-functions} and \ref{sec-coadjoint-orbits}.
Also this section is an illustrative material for previous sections.

Choose any basis of a step $3$ free nilpotent Lie algebra agreed with the graded structure:
$$
\g_1 = \sspan{\{\xi_1,\dots,\xi_r\}}, \qquad \g_2 = \sspan{\{\xi_{12},\dots,\xi_{(r-1)\,r}\}}, \qquad \g_3 = \sspan{\{\xi_{112},\dots,\xi_{r\,(r-1)\,r}\}},
$$
$$
\xi_{ij} = [\xi_i, \xi_j], \qquad \xi_{ijk} = [\xi_i, \xi_{jk}].
$$
Of course, $\xi_{ijk}$ are linearly dependent because of the Jacobi identity. Choose a linearly independent system.

The matrix of the bi-vector $B_{p}$ in this basis reads as shown in Table~\ref{tb-B}.

\begin{table}[h]
\centering
\parbox{.45\linewidth}{
\centering
\begin{tabular}{cccc}
                      &         $\g_1$              &        $\g_2$               &         $\g_3$              \\ \cline{2-4}
\multicolumn{1}{c|}{$\g_1$} & \multicolumn{1}{c|}{$B^{11}_{p}$} & \multicolumn{1}{c|}{$B^{12}_{p}$} & \multicolumn{1}{c|}{$0$} \\ \cline{2-4}
\multicolumn{1}{c|}{$\g_2$} & \multicolumn{1}{c|}{$-(B^{12}_{p})^T$} & \multicolumn{1}{c|}{$0$} & \multicolumn{1}{c|}{$0$} \\ \cline{2-4}
\multicolumn{1}{c|}{$\g_3$} & \multicolumn{1}{c|}{$0$} & \multicolumn{1}{c|}{$0$} & \multicolumn{1}{c|}{$0$} \\ \cline{2-4}
\end{tabular}
\caption{\label{tb-B}The matrix of the bi-vector $B_p$.}
}
\hfil
\parbox{.45\linewidth}{
\centering
\begin{tabular}{cccc}
                      &         $\g_1$              &        $\g_2$               &         $\g_3$              \\ \cline{2-4}
\multicolumn{1}{c|}{$\g_1$} & \multicolumn{1}{c|}{$\g_2$} & \multicolumn{1}{c|}{$\g_3$} & \multicolumn{1}{c|}{$0$} \\ \cline{2-4}
\multicolumn{1}{c|}{$\g_2$} & \multicolumn{1}{c|}{$\g_3$} & \multicolumn{1}{c|}{$0$} & \multicolumn{1}{c|}{$0$} \\ \cline{2-4}
\multicolumn{1}{c|}{$\g_3$} & \multicolumn{1}{c|}{$0$} & \multicolumn{1}{c|}{$0$} & \multicolumn{1}{c|}{$0$} \\ \cline{2-4}
\end{tabular}
\caption{\label{tb-pb}Poisson brackets.}
}
\end{table}

Here the rows and the columns are numerated by the subspaces $\g_1$, $\g_2$, $\g_3$. The matrix $B^{11}_{p}$ is a skew-symmetric $r\times r$ matrix and $B^{12}_{p}$ is an $r \times r(r-1)/2$ matrix. Table~\ref{tb-pb} shows that elements of the matrices $B^{11}_{p}$ and $B^{12}_{p}$ belong to the subspaces $\g_2$ and $\g_3$ respectively.

The basis of the subspace $\g_3$ forms a system of linear Casimir functions.

Now let us construct Casimir functions from the matrix $B^{12}_p$.
Let $\eta_1, \dots, \eta_{d_2}$ be a basis of the subspace $\g_2$, where $d_2 = \dim{\g_2}$.
Add a row consisting of these basis vectors to the matrix $B^{12}_p$.
Then every $(r+1)\times (r+1)$-minor of the resulting matrix is a Casimir function.
Indeed, taking the Poisson bracket of this function with any element $\xi \in \g_1$ we get
a $(r+1)\times (r+1)$-minor of the matrix $B^{12}_p$ with additional row $\{\xi, \eta_1\}, \dots, \{\xi, \eta_{d_2}\}$.
But this row depends linearly on the other rows. So, our minor equals zero.
It is clear that this Casimir function is a polynomial of degree $r+1$ and it is linear in $\eta_1, \dots, \eta_{d_2}$.
The explicit formula for these Casimir functions reads as:
$$
F_i = \sum\limits_{j \in J}{(-1)^{j-i}\eta_j \det{\{\xi_k, \eta_l\}_{k \in \{1,\dots,r\}}^{l \in J \setminus \{j\}}}}, \qquad
i = 1, \dots, d_2 - r, \qquad
J = \{i,\dots,i+r\}.
$$
Furthermore, this gives a constructive proof of Proposition~\ref{prop-Kerb-is-linearCasimir}.
A result of explicit computations for $r=4$ can be found in \nameref{sec-appendix}.

If $\rk{B^{12}_{p}} = r$ we get a complete system of Casimir functions.
Now discuss a special case when $\rk{B^{12}_{p}} < r$ and construct additional functions which determine coadjoint orbits.

Reduce the matrix $B^{12}_{p}$ to a stepwise form with the help of the following operations:
\begin{enumerate*}
\item permutation of columns;
\item multiplication of a row by an element from the subspace $\g_3$ which does not vanish at the point $p$;
\item replacing a row by its sum with another row.
\end{enumerate*}

We will get a stepwise matrix with elements that are polynomials of elements of the subspace $\g_3$.
Note that this stepwise matrix is written in the permuted basis of the subspace $\g_2$ and a basis of the subspace $\g_1$ whose elements are linear combinations of $\xi_1,\dots,\xi_r$ with coefficients from $\g_3$ that are constant on the set $\mathcal{W}$.

To construct a function of the form $i_pI(p)$ such that $i_pI(\cdot) \in \Ker{B^{12}_{p}}$:
\begin{enumerate*}
\item take non-zero rows of the matrix $B^{12}_p$ and add to them the row $\eta_1, \dots, \eta_{d_2}$;
\item take $(\rk{B^{12}_p} + 1)\times(\rk{B^{12}_p} + 1)$-minors of the resulting matrix as above.
\end{enumerate*}

Consider the subspace $\h_1 = (\Image{B^{12}_{p}})^\circ \subset \g_1$ and the skew-symmetric matrix
$C = B^{11}_{p}|_{\h_1}$. We already know a basis of the subspace $\h_1$ which consists of linear combinations of $\xi_1,\dots,\xi_r$ with coefficients from $\g_3$.
Note that we constructed this basis during making the matrix $B^{12}_{p}$ stepwise. For the matrix $B_{p}$ see Table~\ref{tb-special-B}.

\begin{table}[h]
\centering
\parbox{.45\linewidth}{
\centering
\begin{tabular}{cccccc}
\cline{3-4}
                                              & \multicolumn{1}{c|}{}       & \multicolumn{2}{c|}{$\g_1$}                                               &                                        &                          \\ \cline{3-4}
                                              &                             &                                             & $\h_1$                      & $\g_2$                                 & $\g_3$                   \\ \cline{1-1} \cline{3-6}
\multicolumn{1}{|c|}{\multirow{2}{*}{$\g_1$}} & \multicolumn{1}{c|}{}       & \multicolumn{1}{c|}{$\ast$}                 & \multicolumn{1}{c|}{$\ast$} & \multicolumn{1}{c|}{$B^{12}_{p}$} & \multicolumn{1}{c|}{$0$} \\ \cline{3-6}
\multicolumn{1}{|c|}{}                        & \multicolumn{1}{c|}{$\h_1$} & \multicolumn{1}{c|}{$\ast$}                 & \multicolumn{1}{c|}{$C$}      & \multicolumn{1}{c|}{$0$}               & \multicolumn{1}{c|}{$0$} \\ \cline{1-1} \cline{3-6}
                                              & \multicolumn{1}{c|}{$\g_2$} & \multicolumn{1}{c|}{$-(B^{12}_{p})^T$} & \multicolumn{1}{c|}{$0$}    & \multicolumn{1}{c|}{$0$}               & \multicolumn{1}{c|}{$0$} \\ \cline{3-6}
                                              & \multicolumn{1}{c|}{$\g_3$} & \multicolumn{1}{c|}{$0$}                    & \multicolumn{1}{c|}{$0$}    & \multicolumn{1}{c|}{$0$}               & \multicolumn{1}{c|}{$0$} \\ \cline{3-6}
\end{tabular}
\caption{\label{tb-special-B}The matrix of the bi-vector $B_p$ in the special case $\rk{B^{12}_p} < r$.}
}
\hfil
\parbox{.45\linewidth}{
\centering
\begin{tabular}{cccccc}
\cline{4-6}
                                              &                                              & \multicolumn{1}{c|}{}      & \multicolumn{3}{c|}{$\g_1$}                                                             \\ \cline{4-6}
                                              &                                              &                            & \multicolumn{1}{c|}{}       & \multicolumn{2}{c|}{$\h_1$}                               \\ \cline{5-6}
                                              &                                              &                            &                             &                             & $\kk$                       \\ \cline{1-1} \cline{4-6}
\multicolumn{1}{|c|}{\multirow{3}{*}{$\g_1$}} &                                              & \multicolumn{1}{c|}{}      & \multicolumn{1}{c|}{$\ast$} & \multicolumn{1}{c|}{$\ast$} & \multicolumn{1}{c|}{$\ast$} \\ \cline{2-2} \cline{4-6}
\multicolumn{1}{|c|}{}                        & \multicolumn{1}{c|}{\multirow{2}{*}{$\h_1$}} & \multicolumn{1}{c|}{}      & \multicolumn{1}{c|}{$\ast$} & \multicolumn{1}{c|}{$\ast$} & \multicolumn{1}{c|}{$0$}    \\ \cline{4-6}
\multicolumn{1}{|c|}{}                        & \multicolumn{1}{c|}{}                        & \multicolumn{1}{c|}{$\kk$} & \multicolumn{1}{c|}{$\ast$} & \multicolumn{1}{c|}{$0$}    & \multicolumn{1}{c|}{$0$}    \\ \cline{1-2} \cline{4-6}
\end{tabular}
\caption{\label{tb-B11}The matrix $B^{11}_p$ and its kernel.}
}
\end{table}

Consider the skew-symmetric matrix $C$. Bring this matrix to a block-diagonal form and find a basis of its kernel $\kk$, see Table~\ref{tb-B11}.
Let $\zeta_1, \dots, \zeta_r$ be the corresponding basis of the subspace $\g_1$ such that the last $k_2$ vectors form a basis of the subspace $\kk$.
For an element $\gamma \in \kk$ consider the vector $b$ with coordinates $b_i = [\zeta_i, \gamma]$ for $i=1,\dots,r-k_1$ (see Table~\ref{tb-b}) and solve the equation $B^{12}_p\eta = b$ for $\eta \in \g_2$ the following way.

Consider the basis of the subspace $\g_2$ such that its first $r-k_1$ vectors are $[\zeta_i, \gamma]$.
The matrix $D$ is an $(r-k_1)\times(r-k_1)$ fragment of the matrix $B^{12}_p$, see Table~\ref{tb-D}.
Let $b = b_1 + b_2$, where $b_1 \in \Image{D}$.

Find $\eta_1 = D^{-1}(b_1)$ with the help of Cramer's rule. Then each coordinate of $\eta_1$ is a fraction, where the nominator is linear in $b_1$ and the denominator equals $\det{D}$ that is a function polynomial in $\g_3$. Multiply $\eta_1$ and $\gamma$ by this denominator.

Find $\eta_2 \in (B^{12})^{-1}(b_2)$. The coordinates of $\eta_2$ are fractions, where the nominators are linear in $b_2$ and the denominators are polynomials of variables from $\g_3$. Multiply $\eta_1$, $\eta_2$ and $\gamma$ by the common denominator.

\begin{table}[h]
\centering
\parbox{.45\linewidth}{
\centering
\begin{tabular}{ccccccccccccc}
\cline{3-13}
                                              & \multicolumn{1}{c|}{} & \multicolumn{5}{c|}{$\kk$}                                                                                                              & \multicolumn{6}{c|}{$\g_2$}                      \\ \cline{3-13}
                                              &                       &                       &                      & $\gamma$                                  &                       &                      &        &        &        &       &       &       \\ \cline{1-1} \cline{3-13}
\multicolumn{1}{|c|}{\multirow{5}{*}{$\g_1$}} & \multicolumn{1}{c|}{} & \multicolumn{2}{c|}{\multirow{3}{*}{$\ast$}} & \multicolumn{1}{c|}{\multirow{3}{*}{$b$}} & \multicolumn{2}{c|}{\multirow{3}{*}{$\ast$}} & \multicolumn{6}{c|}{\multirow{3}{*}{$B^{12}_p$}} \\
\multicolumn{1}{|c|}{}                        & \multicolumn{1}{c|}{} & \multicolumn{2}{c|}{}                        & \multicolumn{1}{c|}{}                     & \multicolumn{2}{c|}{}                        & \multicolumn{6}{c|}{}                            \\
\multicolumn{1}{|c|}{}                        & \multicolumn{1}{c|}{} & \multicolumn{2}{c|}{}                        & \multicolumn{1}{c|}{}                     & \multicolumn{2}{c|}{}                        & \multicolumn{6}{c|}{}                            \\ \cline{3-13}
\multicolumn{1}{|c|}{}                        & \multicolumn{1}{c|}{} & \multicolumn{5}{c|}{\multirow{2}{*}{$0$}}                                                                                               & \multicolumn{6}{c|}{\multirow{2}{*}{$0$}}        \\
\multicolumn{1}{|c|}{}                        & \multicolumn{1}{c|}{} & \multicolumn{5}{c|}{}                                                                                                                   & \multicolumn{6}{c|}{}                            \\ \cline{1-1} \cline{3-13}
\end{tabular}
\caption{\label{tb-b}Fragment of the matrix $B_p$ and vector $b$.}
}
\hfil
\parbox{.45\linewidth}{
\centering
\begin{tabular}{cccllc}
                                           &                                                & \multicolumn{3}{c}{$[\zeta_1,\gamma], [\zeta_2, \gamma], \dots$} & $\Ker{D}$                                 \\ \cline{2-6}
\multicolumn{1}{c|}{\multirow{3}{*}{\rotatebox[origin=c]{90}{$\Image{D}$, $b_1$}}} & \multicolumn{1}{c|}{$[\zeta_1,\gamma]$}        & \multicolumn{3}{c|}{\multirow{4}{*}{$\ast$}}                     & \multicolumn{1}{c|}{\multirow{4}{*}{$0$}} \\
\multicolumn{1}{c|}{}                      & \multicolumn{1}{c|}{$[\zeta_2,\gamma]$}        & \multicolumn{3}{c|}{}                                            & \multicolumn{1}{c|}{}                     \\
\multicolumn{1}{c|}{}                      & \multicolumn{1}{c|}{\multirow{2}{*}{$\vdots$}} & \multicolumn{3}{c|}{}                                            & \multicolumn{1}{c|}{}                     \\
\multicolumn{1}{c|}{}                      & \multicolumn{1}{c|}{}                          & \multicolumn{3}{c|}{}                                            & \multicolumn{1}{c|}{}                     \\ \cline{2-6}
\multicolumn{1}{c|}{}                      & \multicolumn{1}{c|}{\multirow{2}{*}{$b_2$}}    & \multicolumn{3}{c|}{\multirow{2}{*}{$0$}}                        & \multicolumn{1}{c|}{\multirow{2}{*}{$0$}} \\
\multicolumn{1}{c|}{}                      & \multicolumn{1}{c|}{}                          & \multicolumn{3}{c|}{}                                            & \multicolumn{1}{c|}{}                     \\ \cline{2-6}
\end{tabular}
\caption{\label{tb-D}The matrix $D$ and the vectors $b_1$, $b_2$.}
}
\end{table}

To construct the function $\widehat{I}_{\gamma}$ multiply each coordinate of the element $\gamma$ by the corresponding basis vector of the subspace $\kk$ and take the sum of these products.

Construct the function $\widehat{Q}_{\gamma}$ in two steps.
First, multiply each coordinate of the element $\eta_1$ by the corresponding basis vector of the subspace $\g_2$ and divide the sum of these products by 2.
We get the function $i_pA_{\gamma}(p, p)$.
Second, multiply each coordinate of the element $\eta_2$ by the corresponding basis vector of the subspace $\g_2$ and take the sum of these products.
This is the function $i_pL_{\gamma}(p)$.
The result is the function $\widehat{Q}_{\eta} = i_pA_{\gamma}(p, p) + i_pL_{\gamma}(p)$.

Finally, the function $\widehat{I}_{\gamma} - \widehat{Q}_{\eta}$ is a required function.

\begin{example}
\label{ex-33}
Consider the free Carnot group of rank 3 and step 3.
The corresponding Lie algebra $\g = \g_1 \oplus \g_2 \oplus \g_3$ has the following structure:
$$
\g_1 = \sspan{\{\xi_1, \xi_2, \xi_3\}}, \qquad
\g_2 = \sspan{\{\xi_{12}, \xi_{13}, \xi_{23}\}}, \qquad
\g_3 = \sspan{\{\xi_{ijk}\}},
$$
$$
[\xi_j, \xi_k] = \xi_{jk}, \qquad [\xi_k, \xi_{jk}] = \xi_{ijk}, \qquad j < k,
$$
$$
\xi_{123} - \xi_{213} + \xi_{312} = 0.
$$
For general $p \in \g^*$ the $3\times 3$-matrix $B^{12}_p$ has trivial kernel.
So, we have only 8 independent linear Casimir functions of the form $\g_3$.

Consider now the special case of $\rk{B^{12}_p} = 2$.
After a change of basis we may assume that $\xi_{312} = \xi_{313} = \xi_{323} = 0$.
Consequently, $\xi_{123} = \xi_{213}$.
The one-dimensional subspace $\Ker{B^{12}_p}$ gives us the following Casimir function that is linear on joint level sets of functions of the first type:
$$
\xi_{12}(\xi_{113}\xi_{223}-\xi_{213}\xi_{123}) +
\xi_{13}(\xi_{212}\xi_{123}-\xi_{112}\xi_{223}) +
\xi_{23}(\xi_{112}\xi_{213}-\xi_{113}\xi_{212}).
$$
Next, we have the subspace $h_1 = \sspan{\{\xi_3\}}$ and $\gamma = \xi_3$.
We obtain the following matrix $D$ and vector $b = B^{11}_p(\gamma)$:
$$
D = \left(
\begin{array}{cc}
\xi_{113} & \xi_{123} \\
\xi_{213} & \xi_{223} \\
\end{array}
\right), \qquad
b = \left(
\begin{array}{c}
\xi_{13}\\
\xi_{23}
\end{array}
\right).
$$
If $\det{D} \neq 0$, then solve the equation $D\eta = b$. Multiplying the solution by $\det{D}$ we obtain
$$
\eta^{13} = \xi_{13}\xi_{223} - \xi_{23}\xi_{123}, \qquad \eta^{23} = \xi_{23}\xi_{113} - \xi_{13}\xi_{213}.
$$
Construct the functions
$$
\textstyle
\widehat{Q}_{\gamma} = {{1}\over{2}}(\eta^{13}\xi_{13} + \eta^{23}\xi_{23}), \qquad
\widehat{I}_{\gamma} = \xi_3 \det{D}.
$$
Finally, the additional function that is constant on special coadjoint orbits is
$$
\textstyle
\widehat{I}_{\gamma} - \widehat{Q}_{\gamma} = \xi_3 (\xi_{113}\xi_{223} - \xi_{123}^2) -
{{1}\over{2}}(\xi_{13}^2\xi_{223} - \xi_{13}\xi_{23} (\xi_{123} + \xi_{213}) + \xi_{23}^2\xi_{113}).
$$
The corresponding coadjoint orbit is a direct product of a plane and a hyperbolic paraboloid or an elliptic paraboloid when $\det{D} < 0$ or $\det{D} > 0$, respectively.

Now assume that the matrix $D$ is degenerate. For instance $\xi_{123} = \xi_{223} = 0$, i.e., $\Ker{D} = \sspan{\{\xi_{23}\}}$.
Solving the equation $B^{12}_p(\eta) = b$ we obtain
$$
\eta_1^{13} = \frac{\xi_{13}}{\xi_{113}}, \qquad \eta_1^{12} = \eta_1^{23} = 0,
$$
$$
\eta_2^{12} = \frac{\xi_{23}}{\xi_{212}}, \qquad \eta_2^{13} = -\frac{\xi_{23}\xi_{112}}{\xi_{113}\xi_{212}}, \qquad \eta_2^{23} = 0.
$$
Multiply $\eta_1, \eta_2$ and $\gamma$ by $\xi_{113}\xi_{212}$. We obtain
$$
\bar{\eta}_1^{13} = \xi_{13}\xi_{212}, \qquad
\bar{\eta}_2^{12} = \xi_{23}\xi_{113}, \qquad
\bar{\eta}_2^{13} = -\xi_{23}\xi_{112}.
$$
Taking ${{1}\over{2}}\bar{\eta}_1^{13}\xi_{13} + \bar{\eta}_2^{12}\xi_{12} + \bar{\eta}_2^{13}\xi_{13}$ we obtain $\widehat{Q}_{\gamma}$.
Next, $\widehat{I}_{\gamma} = \xi_3\xi_{113}\xi_{212}$. Finally, we get
$$
\textstyle
\widehat{I}_{\gamma} - \widehat{Q}_{\gamma}  = \xi_3\xi_{113}\xi_{212} - {{1}\over{2}}\xi_{13}^2\xi_{212} - \xi_{12}\xi_{23}\xi_{113} + \xi_{13}\xi_{23}\xi_{112}.
$$
The corresponding coadjoint orbit is a parabolic cylinder (remind that $\xi_{23}$ is constant on the joint level set $\mathcal{W}$ of linear Casimir functions).

See Table~\ref{tb-orbits-3-3} for the coadjoint orbits types depending on the Poisson bi-vector.

\begin{table}[h!]
\centering
\begin{tabular}{ccclc}
\multicolumn{3}{c}{Type of Poisson bi-vector}                                                                                   & \multicolumn{1}{c}{Orbit $\mathcal{O}$ type} & $\dim{\mathcal{O}}$      \\ \hline
\multicolumn{3}{|c|}{$\rk{B^{12}_p} = 3$}                                                                                       & \multicolumn{1}{l|}{affine subspace}         & \multicolumn{1}{c|}{$6$} \\ \hline
\multicolumn{1}{|c|}{\multirow{4}{*}{$\rk{B^{12}_p} = 2$}} & \multirow{2}{*}{$\rk{D} = 2$} & \multicolumn{1}{c|}{$\det{D} < 0$} & \multicolumn{1}{l|}{$\R^2 \times$hyperbolic paraboloid}   & \multicolumn{1}{c|}{$4$} \\ \cline{4-5}
\multicolumn{1}{|c|}{}                                     &                               & \multicolumn{1}{c|}{$\det{D} > 0$} & \multicolumn{1}{l|}{$\R^2 \times$elliptic paraboloid}     & \multicolumn{1}{c|}{$4$} \\ \cline{2-5}
\multicolumn{1}{|c|}{}                                     & \multicolumn{2}{c|}{$\rk{D} = 1$}                                  & \multicolumn{1}{l|}{parabolic cylinder}      & \multicolumn{1}{c|}{$4$} \\ \cline{2-5}
\multicolumn{1}{|c|}{}                                     & \multicolumn{2}{c|}{$\rk{D} = 0$}                                  & \multicolumn{1}{l|}{affine subspace}         & \multicolumn{1}{c|}{$4$} \\ \hline
\multicolumn{1}{|c|}{\multirow{2}{*}{$\rk{B^{12}_p} = 1$}} & \multicolumn{2}{c|}{$\rk{B^{11}_p|_{\h_1}} = 2$}                   & \multicolumn{1}{l|}{affine subspace}         & \multicolumn{1}{c|}{$4$} \\ \cline{2-5}
\multicolumn{1}{|c|}{}                                     & \multicolumn{2}{c|}{$\rk{B^{11}_p|_{\h_1}} = 0$}                   & \multicolumn{1}{l|}{affine subspace}         & \multicolumn{1}{c|}{$2$} \\ \hline
\multicolumn{1}{|c|}{\multirow{2}{*}{$\rk{B^{12}_p} = 0$}} & \multicolumn{2}{c|}{$\rk{B^{11}_p} = 2$}                           & \multicolumn{1}{l|}{affine subspace}         & \multicolumn{1}{c|}{$2$} \\ \cline{2-5}
\multicolumn{1}{|c|}{}                                     & \multicolumn{2}{c|}{$\rk{B^{11}_p} = 0$}                           & \multicolumn{1}{l|}{point}                   & \multicolumn{1}{c|}{$0$} \\ \hline
\end{tabular}
\caption{\label{tb-orbits-3-3}Coadjoint orbits of free Carnot group of step 3 and rank 3.}
\end{table}

\end{example}

\section{\label{sec-s4}Free Carnot groups of step 4}

Let us describe Casimir functions for free Carnot groups of step 4.
By formula~\eqref{eq-grad-component-dimension} we have $\dim{\g_3} = (r^3 - r)/3$.
We see that $\dim{\g_3} \geqslant r$ for $r \geqslant 2$.
So, as we know from Theorem~\ref{th-casimirs-r-s} there are linear Casimir functions corresponding to $\g_4$ and Casimir functions constructed from the subspace $\Ker{B^{13}_p}$,
where the map $B^{13}_p : \g_3 \rightarrow \g_1^*$ is such that $B^{13}_p(\eta) = B_p(\, \cdot \, , \eta)$ for $\eta \in \g_3$.

Consider the map $B^{(12)3}_p : \g_3 \rightarrow (\g_1 \oplus \g_2)^*$, where $B^{(12)3}_p = B^{13}_p\oplus 0$.
For almost all $p \in \g^*$ we have $\Image{B^{(12)3}_p} = \g_1^*$, then $\g_2 = (\Image{B^{(12)3}_p})^{\circ}$.
Next, for $\gamma \in \Ker{B_p|_{\g_2}}$ we have $B_p(\gamma) \in \Image{B^{(12)3}_p}$, as in Proposition~\ref{prop-preimageB12}.
Moreover, the coordinates of the vector $B_p(\gamma)$ are in the subspace $[\g_1, \gamma] \subset \g_3$ and the coefficients of the map $B_p|_{\g_2}$ are in the subspace
$[\g_2, \g_2] \subset \g_4$. So, we can construct a quadratic Casimir function in a way similar to Theorem~\ref{th-coadjoint-orbits}.
We obtain the following theorem.

\begin{theorem}
The full system of Casimir functions for a free Carnot group of step \emph{4} consists of\\
\emph{(1)} linear functions of the form $\g_4$\emph{;} \\
\emph{(2)} functions constructed with the help of Proposition~\emph{\ref{prop-linearCasimirs}} from the subspace $\Ker{B^{13}_p}$,
these functions are linear on joint level sets of functions~\emph{(1)}\emph{;}\\
\emph{(3)} functions of the form $\widehat{I}_{\gamma} - \widehat{Q}_{\gamma}$, where $\gamma \in \Ker{B_p|_{\g_2}}$, $\eta \in (B^{13}_p)^{-1}B_p(\gamma)$ and
$\widehat{I}_{\gamma}$, $\widehat{Q}_{\gamma}$ are defined as in Section~\emph{\ref{sec-coadjoint-orbits}},
these functions are quadratic on joint level sets of functions~\emph{(1, 2)}.
\end{theorem}

%\begin{comment}
\begin{proof}
The proof is similar to the proof of Theorem~\ref{th-coadjoint-orbits}.
Note that in this case linear functions of the form $\h_2 = [\g_2, \g_2] \subset \g_4$ are constant on joint level sets of functions~(1).
Further, the values of the linear function $L_{\gamma} : \g^* \rightarrow S(\g_4 \oplus \Ker{D})$ are linear on variables from the space $\Ker{D}$.
This implies that the functions $\widehat{I}_{\gamma} - \widehat{Q}_{\gamma}$  are quadratic on joint level sets of functions~(1, 2).
\end{proof}
%\end{comment}

\begin{example}
\label{ex-2358}
Consider the free Carnot group of step 4 and rank 2. The corresponding Lie algebra has the following structure:
$$
\g_1 = \sspan{\{\xi_1, \xi_2\}}, \qquad \g_2 = \sspan{\{\xi_{12}\}}, \qquad \g_3 = \sspan{\{\xi_{112}, \xi_{212}\}},
$$
$$
\g_4 = \sspan{\{\xi_{1112}, \xi_{1212}, \xi_{2212}\}},
$$
$$
[\xi_1, \xi_2] = \xi_{12}, \qquad [\xi_1, \xi_{12}] = \xi_{112}, \qquad [\xi_2, \xi_{12}] = \xi_{212},
$$
$$
[\xi_1, \xi_{112}] = \xi_{1112}, \qquad [\xi_1, \xi_{212}] = [\xi_2, \xi_{112}] = \xi_{1212}, \qquad [\xi_2, \xi_{212}] = \xi_{2212}.
$$
It is easy to check (see~\cite{lokutsievskii-sachkov}) that there are three linear Casimir functions and a function that is quadratic on joint level sets of linear ones:
$$
\xi_{1112}, \qquad \xi_{1212}, \qquad \xi_{2212},
$$
\begin{equation*}
\textstyle
\xi_{12}(\xi_{1112}\xi_{2212} - \xi_{1212}^2) - {{1}\over{2}}\xi_{2212}\xi_{112}^2 - {{1}\over{2}}\xi_{1112}\xi_{212}^2 + \xi_{1212}\xi_{112}\xi_{212}.
\end{equation*}

Let us explain how can we get these Casimir functions. Here we obtain $\Ker{B^{13}_p} = 0$. Next, since $\g_2 = \sspan{\{\xi_{12}}\}$ the form $B_p|_{\g_2}$ is zero.
For $\gamma = \xi_{12}$ construct $\eta \in \g_3$ such that $B_p(\gamma) = B^{13}_p(\eta)$. It is easy to see that in the basis $\xi_{112}, \xi_{212}$ the coordinates of $\eta$ are
$$
\frac{\xi_{112}\xi_{2212} - \xi_{212}\xi_{1212}}{\xi_{1112}\xi_{2212}-\xi_{1212}^2}, \qquad
\frac{\xi_{212}\xi_{1112} - \xi_{112}\xi_{1212}}{\xi_{1112}\xi_{2212}-\xi_{1212}^2}.
$$
Multiply $\eta$ and $\gamma$ by the common denominator of these two fractions.
Then $\widehat{I}_{\gamma} = \xi_{12} (\xi_{1112}\xi_{2212}-\xi_{1212}^2)$.
Multiply each coordinate $\eta$ by the corresponding basis vector and take the sum of these products.
Dividing the resulting quadratic expression (with respect to the variables $\xi_{112}$, $\xi_{212}$) by 2 we get the function $\widehat{Q}_{\eta}$.
\end{example}

Consider now the special case when $\Image{B^{13}_p} \neq \g_1^*$. A joint level set of functions~(1--3) consists of coadjoint orbits of low dimensions.
The method of Section~\ref{sec-coadjoint-orbits} does not work for the task of description of these coadjoint orbits.
Indeed, the subspace $\h_1 = (\Image{B^{13}_p})^{\circ}$ contains elements of the subspace $\g_1$.
Consequently, the coefficients of the map $B_p|_{\h_1}$ depend on elements of the subspace $[\g_1, \g_1] = \g_2$ and these elements are not constant
(here we do not have an analog of Proposition~\ref{prop-h2-are-constant-on-O}).

Similar difficulties appear in the case of free Carnot group of step $s \geqslant 5$ even for description of general coadjoint orbits.

\section{\label{sec-control}Application to control theory}

Here we discuss an application of Theorem~\ref{th-coadjoint-orbits} to a behavior of extremal controls in time-optimal left-invariant control problems.

Let $G$ be a three-step free Carnot group of rank $r$.
The corresponding Lie algebra $\g$ is generated by $\g_1 = \sspan{\{\xi_1,\dots,\xi_r\}}$.
Consider the left-invariant vector fields $X_i(g) = dL_g \xi_i$ for $i=1,\dots,r$, where $L_g$ is a left-shift by an element $g \in G$.
Assume that $U \subset \R^r$ is a convex compact set containing the origin in its interior. Consider the following time-optimal problem:
\begin{equation}
\label{eq-control}
\begin{array}{l}
\dot{g} = \sum\limits_{i=1}^{r}{u_iX_i(g)}, \qquad g \in G, \quad u = (u_1,\dots,u_r) \in U, \\
g(0) = \id, \quad g(t_1) = g_1 \in G, \\
t_1 \rightarrow \min.\\
\end{array}
\end{equation}
Note that if $U = -U$ we obtain a sub-Finsler problem~\cite{barilari-boscain-ledonne-sigalotti,berestovskii}. In particular, if $U$ is an ellipsoid we have a sub-Riemannian problem~\cite{agrachev-barilary-boscain}.

The Pontryagin maximum principle~\cite{pontryagin,argachev-sachkov} gives necessary conditions of optimality. Normal extremal trajectories are projections of trajectories of a Hamiltonian vector field $\vec{H}$ on $T^*G$ called \emph{extremals}. The corresponding Hamiltonian reads as
$$
H(h_1,\dots,h_r) = \max\limits_{v \in U}{\sum\limits_{i=1}^r{v_ih_i}},
$$
where the functions $h_i = \langle\,\cdot\, , X_i\rangle$ are linear on the fibers of the cotangent bundle $T^*G$.
Any extremal is defined by its initial point at $T^*_{\id}G = \g^*$ called \emph{an initial momentum}.
The vertical part of the Hamiltonian vector field is determined by the equation
\begin{equation}
\label{eq-vert-subsystem}
\dot{h} = \{H, h\}.
\end{equation}

The simplest integrable case is the case of an extremal with an initial momentum lying on a two-dimensional coadjoint orbit.
Indeed, this coadjoint orbit is an invariant sub-manifold of equation~\eqref{eq-vert-subsystem} and the Hamiltonian $H$ is a first integral of this equation.
It turns out that two-dimensional coadjoint orbits of the Lie group $G$ are arranged the same way as coadjoint orbits of some simplest Carnot groups.
Let us recall the corresponding definitions.

\begin{definition}
\label{def-heisenberg}
The free nilpotent Lie algebra of rank 2 and step 2 is called \emph{the Heisenberg algebra}. The corresponding connected and simply connected Lie group $H_3$ is called \emph{the Heisenberg group}.
\end{definition}

\begin{remark}
\label{rem-heisenberg-orbits}
The Heisenberg algebra is the simplest nilpotent Lie algebra of rank 2. It is spanned by the elements
$\zeta_1$, $\zeta_2$, $\zeta_{12} = [\zeta_1, \zeta_2]$.

Affine planes $\zeta_{12} = \const \neq 0$ are two-dimensional coadjoint orbits of the group $H_3$. Points of the plane $\zeta_{12} = 0$ are zero-dimensional coadjoint orbits.
\end{remark}

\begin{definition}
\label{def-engel}
\emph{The Engel algebra} is a Lie algebra that is spanned by
$$
\zeta_1, \qquad \zeta_2, \qquad \zeta_{12} = [\zeta_1, \zeta_2], \qquad \zeta_{112} = [\zeta_1, \zeta_{12}].
$$
Other commutators of these basis elements equal zero.
The corresponding connected and simply connected Lie group $E$ is called \emph{the Engel group}.
\end{definition}

\begin{remark}
\label{rem-engel-orbits}
The Engel algebra is the simplest nilpotent Lie algebra of rank 3, but it is not free.
Coadjoint orbits of the group $E$ can be two-dimensional or zero-dimensional.
There are two types of two-dimensional coadjoint orbits:
parabolic cylinders $\zeta_2\zeta_{112} - {{1}\over{2}}\zeta_{12}^2 = \const$ for $\zeta_{112} \neq 0$ and
affine planes $\zeta_{112} = 0$, $\zeta_{12} = \const \neq 0$.
Points of the plane $\zeta_{112} = \zeta_{12} = 0$ are zero-dimensional coadjoint orbits.
See~\cite{ardentov-sachkov-engel} for details.
\end{remark}

\begin{theorem}
\label{th-2dim-trajectories}
Consider a two-dimensional coadjoint orbit $(\Ad^*{G})p \subset \g^*$.
There exists an invariant affine subspace $(\Ad^*{G})p \subset \mathcal{A} \subset \g^*$ and a Lie group $L$ acting on $\mathcal{A}$ such that\\
\emph{(1)} the orbits of the group $L$ on $\mathcal{A}$ coincide with coadjoint orbits of the group $G$;\\
\emph{(2)} the action of the group $L$ on $\mathcal{A}$ is isomorphic to the coadjoint action of the group $L$;\\
\emph{(3)} if $B^{12}_p = 0$, then $L \simeq H_3$;
           if $B^{12}_p \neq 0$, then $L \simeq E$.
\end{theorem}

\begin{proof}
According to Remark~\ref{rem-dim-of-coadjoint-orbits} if $r \geqslant 3$, then the dimension of the coadjoint orbit equals $2r - (k_1+k_2) = (r-k_1) + (r-k_2)$.
This number is equal to 2 in the following cases.

1. Case $r - k_1 = 0$ and $r - k_2 = 2$. This means that $B^{12}_p = 0$ and $\rk{B^{11}_p} = 2$. So, functions~(3) in Theorem~\ref{th-coadjoint-orbits} are linear on $\W$.
In this case the coadjoint orbit is an affine subspace.

Let $\zeta_1, \zeta_2 \in \g_1$ be such that $B^{11}_p$ restricted to the subspace $\sspan{\{\zeta_1, \zeta_2\}}$ is not degenerate.
Let $\mathcal{A}$ be a joint level set of functions~(1--3) from Theorem~\ref{th-coadjoint-orbits}
of the forms $\g_3$, $\g_2 / \sspan{\{\zeta_{12}\}}$ and $\Ker{B^{11}_p|_{\g_1}}$, where $\zeta_{12} = [\zeta_1, \zeta_2]$.
Linear functions $\zeta_1, \zeta_2$ restricted to $\mathcal{A}$ generate (with the help of the Poisson bracket) the Heisenberg algebra.
The corresponding Heisenberg group $L = H_3$ acts on $\mathcal{A}$, this action is isomorphic to the coadjoint action of the group $H_3$.

2. Case $r - k_1 = 1$ and $r - k_2 = 1$. This means that $\dim{\h_1} = \dim{\g_1} - 1$ and $B^{11}_p|_{\h_1} = 0$.
It follows that $\rk{B^{11}_p} = 2$ for almost all $p \in \W$. Indeed, otherwise $\{\g_1, \g_1\}|_{\W} = 0$ and $B^{12}_p = 0$.

Let $\zeta_1, \zeta_2 \in \g_1$ be such that $\sspan{\{\zeta_1, \zeta_2\}}$ is an invariant subspace of $B^{11}_p$ and $B^{11}_p$ is not degenerate on it.
Then functions~(3) in Theorem~\ref{th-coadjoint-orbits} are linear on $\W$
or have the form $\zeta_2\zeta_{112} - {{1}\over{2}}\zeta_{12}^2$ (see their construction in Proposition~\ref{prop-eta-gives-quadratic}),
where $\zeta_{12} = [\zeta_1, \zeta_2]$ and $\zeta_{112} = [\zeta_1, \zeta_{12}]$.
So, the coadjoint orbit is a parabolic cylinder.

Let $\mathcal{A}$ be a joint level set of functions~(1--3) from Theorem~\ref{th-coadjoint-orbits}
of the forms $\g_3 / \sspan{\{\zeta_{112}\}}$, $\g_2 / \sspan{\{\zeta_{12}\}}$ and $\h_1 / \sspan{\{\zeta_2\}}$.
Note, that $\zeta_1|_{\mathcal{A}}$, $\zeta_2|_{\mathcal{A}}$ generate (with the help of the Poisson bracket) the Engel algebra.
The action of the corresponding Engel group $L = E$ on $\mathcal{A}$ is isomorphic to the coadjoint action of the group $E$.

3. Case $r - k_1 = 2$ and $r - k_2 = 0$. It follows that $r = k_2 > k_1 = r - 2$ in contradiction with $k_2 \leqslant k_1$.

It follows from the construction in cases~1--2 that the Poisson structure on the affine subspace $\mathcal{A}$ coincides with the Poisson structure on the Lie coalgebra corresponding to the group $L$. So, the orbits of the group $L$ on $\mathcal{A}$ coincide with coadjoint orbits of the group $G$ on $\mathcal{A}$.

It remains to consider the case $r = 2$. It is easy to see directly that in this case of the Cartan group (see Example~\ref{ex-Cartan-group}) we have the same situation as in cases~1--2 above.
\end{proof}

\begin{remark}
For the Engel group the normal extremal controls were investigated by A.\,A.~Ardentov and Yu.\,L.~Sachkov~\cite{ardentov-sachkov-engel} for a one-parametric family of sub-Finsler problems. They considered a square rotated by an arbitrary angle as an indicatrix of the sub-Finsler structure.
Note that in $l_{\infty}$ sub-Finsler problem on the Cartan group~\cite{ardentov-ledonne-sachkov} the phase portrait of the vertical subsystem splits to that one-parametric family of corresponding phase portraits on the Engel group.
\end{remark}

\begin{proposition}
Assume that the set of controls $U$ of problem~\emph{\eqref{eq-control}} is strictly convex.
Then any phase curve of the vertical subsystem~\emph{\eqref{eq-vert-subsystem}} with an initial momentum from a two-dimensional coadjoint orbit is a regular curve or a fixed point.
\end{proposition}

\begin{proof}
The maximized Hamiltonian $H$ of the Pontryagin maximum principle is $C^1$-smooth on $\R^r \setminus \{0\}$.
The vertical part of the corresponding Hamiltonian system reads as
\begin{equation}
\label{eq-vertpart}
\dot{p} = - B_p \nabla H, \qquad p \in \g^*,
\end{equation}
where $B_p$ is the Poisson bi-vector on the coalgebra $\g^*$.

Let $p \in \g^*$ be such that the corresponding coadjoint orbit $(\Ad^*{G})p$ is two-dimensional.
Since problem~\eqref{eq-control} is homogeneous it is sufficient to assume that $H(p) = 1$.
Consider $\Gamma = H^{-1}(1) \cap (\Ad^*{G})p$.
Note that $\nabla H$ is a transversal to the surface $H^{-1}(1)$ and $\Ker{B_p}$ is a transversal subspace to the coadjoint orbit $(\Ad^*{G})p$.
There are two cases. If $\nabla H \notin \Ker{B_p}$, then $\Gamma$ is a regular curve in some neighbourhood of the point $p$.
If $\nabla H \in \Ker{B_p}$, then from equation~\eqref{eq-vertpart} we obtain that $p$ is a stationary point.
\end{proof}

As an extension of Yu.\,L.~Sachkov's results~\cite{sachkov-subfinsler-2-step,sachkov-two-step} for free Carnot groups of step 2 to step 3 we get Corollaries~\ref{crl-periodic-or-constant-extremal-control}--\ref{crl-summary-control}.

\begin{corollary}
\label{crl-periodic-or-constant-extremal-control}
Assume that the set of controls $U$ is strictly convex.
Consider a two-dimensional coadjoint orbit $(\Ad^*{G})p$ such that $B^{12}_p = 0$.
Extremal controls of problem~\emph{\eqref{eq-control}} corresponding to normal extremals with initial momenta in the orbit $(\Ad^*{G})p$
are periodic or constant.
\end{corollary}

\begin{proof}
The proof is quite similar to the proof in the case of free Carnot group of step 2, for details see~\cite{sachkov-two-step}.

From the proof of Theorem~\ref{th-2dim-trajectories} we know that the orbit $\orbit = (\Ad^*{G})p$ is an affine subspace (see case~1).
Let $\mathcal{L}$ be a joint level set of linear Casimir functions that correspond to elements of the type $\g_2 \oplus \g_3$.
It is sufficient to consider extremals with initial momenta on the level surface $\mathcal{C} = \{p \in \mathcal{L} \subset \g^* \ | \ H(p) = 1 \}$
that is the polar set for the set of controls $U$. The surface $\mathcal{C}$ is convex and compact.
Consider the set $\Gamma = \mathcal{C} \cap \orbit$, where $\orbit$ is a two-dimensional coadjoint orbit.

We claim that either any point of $\Gamma$ is a stationary point of equation~\eqref{eq-vertpart} or $\Gamma$ is a regular curve without stationary points.
Indeed, consider the problem of minimization for the $C^1$-smooth function $H$ on the affine plane $\orbit$.
Since $\nabla H \in \Ker{B_p}$ is necessary and sufficient condition for minimum in this case, then
this condition should be satisfied or unsatisfied for all points of $\Gamma$ simultaneously.

If for some point $p \in \Gamma$ we have $\nabla H \in \Ker{B_p}$, then every point of $\Gamma$ is a stationary point of equation~\eqref{eq-vertpart}.

If $\nabla H \notin \Ker{B_p}$ for some point $p \in \Gamma$, then the set $\Gamma$ is a convex planar curve without stationary points of equation~\eqref{eq-vertpart}.
In this case the solution is periodic.

Consequently, the corresponding extremal control is constant or periodic since $u = \nabla H$, see~\cite{rockafellar}.
\end{proof}

\begin{corollary}
\label{crl-periodic-constant-or-asymptotically-constant-extremal-control}
Assume that the set of controls $U$ is strictly convex and the boundary of the polar for the set of controls $\partial U^{\circ}$ is $W^2_{\infty}$-smooth.
Consider a two-dimensional coadjoint orbit $(\Ad^*{G})p$ such that $B^{12}_p \neq 0$.
Extremal controls of problem~\emph{\eqref{eq-control}} corresponding to normal extremals with initial momenta in the orbit $(\Ad^*{G})p$
are periodic, constant or asymptotically constant (have constant limits at infinity).
\end{corollary}

\begin{proof}
It follows from Theorem~\ref{th-2dim-trajectories} that there exists an affine subspace $\mathcal{A} \subset \g^*$ such that our two-dimensional coadjoint orbit is an orbit of the Engel group on $\mathcal{A}$. Moreover, $H|_{\mathcal{A}}$ is a first integral of equation~\eqref{eq-vert-subsystem}.
Note, that $H|_{\mathcal{A}} = 1$ defines a cylinder over $U^{\circ} \cap \mathcal{A}$.
Consider the section $U^{\circ} \cap \mathcal{A}$, its boundary is $W^2_{\infty}$-smooth.
Apply Proposition~4 from paper~\cite{ardentov-lokutsievskiy-sachkov} to this two-dimensional set.
We get that the conditions of classical Picard theorem are satisfied for the equation~\eqref{eq-vert-subsystem} on $H^{-1}(1) \cap (\Ad^*{G})p$.
So, a trajectory cannot get to a stationary point at a finite time.
Consequently, the corresponding extremal controls can be periodic, constant or asymptotically constant.
\end{proof}

\begin{corollary}
\label{crl-summary-control}
Consider time-optimal problem~\emph{\eqref{eq-control}} on a $3$-step free Carnot group.
Assume that the set of controls $U$ is strictly convex and $\partial U^{\circ}$ is $W^2_{\infty}$-smooth.
Then extremal controls corresponding to normal extremals with initial momenta in a two-dimensional coadjoint orbit
are periodic, constant or asymptotically constant.
\end{corollary}

\begin{proof}
Immediately follows from Corollaries~\ref{crl-periodic-or-constant-extremal-control}, \ref{crl-periodic-constant-or-asymptotically-constant-extremal-control}.
\end{proof}

\begin{corollary}
\label{crl-optimality-Heisenberg}
Consider an extremal trajectory for problem~\eqref{eq-control} with an initial momentum in a two-dimensional coadjoint orbit $\orbit$ such that the corresponding extremal control is periodic. This extremal trajectory cannot be optimal after the period of the control.
\end{corollary}

\begin{proof}
If the extremal control is periodic, then $(U^{\circ} \cap \mathcal{A}) \cap \interior{U^{\circ}} \neq \emptyset$,
where $\interior{U^{\circ}}$ is the interior of the polar $U^{\circ}$.
Put an origin of the space $\mathcal{A}$ inside the set $U^{\circ} \cap \mathcal{A}$ and
consider the corresponding polar set $U_{\mathcal{A}} = (U^{\circ} \cap \mathcal{A})^{\circ} \subset \sspan{\{\zeta_1, \zeta_2\}}$.
The extremal trajectories of problem~\eqref{eq-control} with initial momenta from the orbit $\orbit$ are the extremal trajectories for the time-optimal problem on the Lie group $L$ (the Heisenberg group or the Engel group) with the control set $U_{\mathcal{A}}$. If an extremal trajectory is not optimal for this problem on the group $L$, then this trajectory cannot be optimal for problem~\eqref{eq-control}.
In particular, in the case of the Heisenberg group or the Engel group if the time is greater than the period of a periodic extremal control, then the corresponding extremal trajectory is not optimal~\cite{berestovskii-heisenberg,berestovskii-zubareva-engel}, consequently this trajectory is not optimal for problem~\eqref{eq-control}.
\end{proof}

\section*{\label{sec-appendix}Appendix}

Here we give Casimir functions computed according to the algorithm from Section~\ref{sec-algorithm} with the help of \texttt{Wolfram Mathematica} for the free 3-step Carnot group of rank 4.

\begin{example}
\label{ex-s3r4}
Consider the free Carnot group of step 3 and rank 4. This group has 20 linear Casimir functions that correspond to the subspace $\g_3$ and 2 Casimir functions $C_1, C_2$ of degree 5 that correspond to the subspace $\Ker{B^{12}_p}$. It follows that general coadjoint orbits are 8-dimensional affine subspaces.
\begin{small}
\begin{equation*}
\begin{array}{lrr}
C_1 = & \xi_{24} & (\xi_{123} \xi_{214} \xi_{313} \xi_{412} - \xi_{114} \xi_{223} \xi_{313} \xi_{412} - \xi_{123} \xi_{213} \xi_{314} \xi_{412} +
      \xi_{113} \xi_{223} \xi_{314} \xi_{412} + \\
      & & \xi_{114} \xi_{213} \xi_{323} \xi_{412} - \xi_{113} \xi_{214} \xi_{323} \xi_{412} -
      \xi_{123} \xi_{214} \xi_{312} \xi_{413} + \xi_{114} \xi_{223} \xi_{312} \xi_{413} + \\
      & & \xi_{123} \xi_{212} \xi_{314} \xi_{413} - \xi_{112} \xi_{223} \xi_{314} \xi_{413} - \xi_{114} \xi_{212} \xi_{323} \xi_{413} + \xi_{112} \xi_{214} \xi_{323} \xi_{413} + \\
     &  & \xi_{123} \xi_{213} \xi_{312} \xi_{414} - \xi_{113} \xi_{223} \xi_{312} \xi_{414} - \xi_{123} \xi_{212} \xi_{313} \xi_{414} +
      \xi_{112} \xi_{223} \xi_{313} \xi_{414} + \\
      & & \xi_{113} \xi_{212} \xi_{323} \xi_{414} - \xi_{112} \xi_{213} \xi_{323} \xi_{414} -
    \xi_{114} \xi_{213} \xi_{312} \xi_{423} + \xi_{113} \xi_{214} \xi_{312} \xi_{423} +  \\
    & & \xi_{114} \xi_{212} \xi_{313} \xi_{423} - \xi_{112} \xi_{214} \xi_{313} \xi_{423} - \xi_{113} \xi_{212} \xi_{314} \xi_{423} + \xi_{112} \xi_{213} \xi_{314} \xi_{423}) \\
& - \xi_{23} & (\xi_{124} \xi_{214} \xi_{313} \xi_{412} - \xi_{114} \xi_{224} \xi_{313} \xi_{412} -
    \xi_{124} \xi_{213} \xi_{314} \xi_{412} + \xi_{113} \xi_{224} \xi_{314} \xi_{412} + \\
    & & \xi_{114} \xi_{213} \xi_{324} \xi_{412} - \xi_{113} \xi_{214} \xi_{324} \xi_{412} - \xi_{124} \xi_{214} \xi_{312} \xi_{413} + \xi_{114} \xi_{224} \xi_{312} \xi_{413} +\\
  &  &  \xi_{124} \xi_{212} \xi_{314} \xi_{413} - \xi_{112} \xi_{224} \xi_{314} \xi_{413} - \xi_{114} \xi_{212} \xi_{324} \xi_{413} +
   \xi_{112} \xi_{214} \xi_{324} \xi_{413} + \\
   & & \xi_{124} \xi_{213} \xi_{312} \xi_{414} - \xi_{113} \xi_{224} \xi_{312} \xi_{414} - \xi_{124} \xi_{212} \xi_{313} \xi_{414} + \xi_{112} \xi_{224} \xi_{313} \xi_{414} + \\
   & & \xi_{113} \xi_{212} \xi_{324} \xi_{414} - \xi_{112} \xi_{213} \xi_{324} \xi_{414} - \xi_{114} \xi_{213} \xi_{312} \xi_{424} + \xi_{113} \xi_{214} \xi_{312} \xi_{424} +\\
    &  &\xi_{114} \xi_{212} \xi_{313} \xi_{424} - \xi_{112} \xi_{214} \xi_{313} \xi_{424} - \xi_{113} \xi_{212} \xi_{314} \xi_{424} +
    \xi_{112} \xi_{213} \xi_{314} \xi_{424}) + \\
 &\xi_{14}& (\xi_{124} \xi_{223} \xi_{313} \xi_{412} - \xi_{123} \xi_{224} \xi_{313} \xi_{412} -
    \xi_{124} \xi_{213} \xi_{323} \xi_{412} + \xi_{113} \xi_{224} \xi_{323} \xi_{412} + \\ & & \xi_{123} \xi_{213} \xi_{324} \xi_{412} -
    \xi_{113} \xi_{223} \xi_{324} \xi_{412} - \xi_{124} \xi_{223} \xi_{312} \xi_{413} + \xi_{123} \xi_{224} \xi_{312} \xi_{413} +\\
    &  &\xi_{124} \xi_{212} \xi_{323} \xi_{413} - \xi_{112} \xi_{224} \xi_{323} \xi_{413} - \xi_{123} \xi_{212} \xi_{324} \xi_{413} +
    \xi_{112} \xi_{223} \xi_{324} \xi_{413} + \\ & & \xi_{124} \xi_{213} \xi_{312} \xi_{423} - \xi_{113} \xi_{224} \xi_{312} \xi_{423} -
    \xi_{124} \xi_{212} \xi_{313} \xi_{423} + \xi_{112} \xi_{224} \xi_{313} \xi_{423} + \\ & & \xi_{113} \xi_{212} \xi_{324} \xi_{423} -
    \xi_{112} \xi_{213} \xi_{324} \xi_{423} - \xi_{123} \xi_{213} \xi_{312} \xi_{424} + \xi_{113} \xi_{223} \xi_{312} \xi_{424} +\\
    &  &\xi_{123} \xi_{212} \xi_{313} \xi_{424} - \xi_{112} \xi_{223} \xi_{313} \xi_{424} - \xi_{113} \xi_{212} \xi_{323} \xi_{424} +
    \xi_{112} \xi_{213} \xi_{323} \xi_{424}) \\
&- \xi_{13} &(\xi_{124} \xi_{223} \xi_{314} \xi_{412} - \xi_{123} \xi_{224} \xi_{314} \xi_{412} -
        \xi_{124} \xi_{214} \xi_{323} \xi_{412} + \xi_{114} \xi_{224} \xi_{323} \xi_{412} + \\ & & \xi_{123} \xi_{214} \xi_{324} \xi_{412} -
       \xi_{114} \xi_{223} \xi_{324} \xi_{412} - \xi_{124} \xi_{223} \xi_{312} \xi_{414} + \xi_{123} \xi_{224} \xi_{312} \xi_{414} +\\
         &  & \xi_{124} \xi_{212} \xi_{323} \xi_{414} - \xi_{112} \xi_{224} \xi_{323} \xi_{414} - \xi_{123} \xi_{212} \xi_{324} \xi_{414} +
         \xi_{112} \xi_{223} \xi_{324} \xi_{414} + \\ & & \xi_{124} \xi_{214} \xi_{312} \xi_{423} - \xi_{114} \xi_{224} \xi_{312} \xi_{423} -
         \xi_{124} \xi_{212} \xi_{314} \xi_{423} + \xi_{112} \xi_{224} \xi_{314} \xi_{423} + \\ & & \xi_{114} \xi_{212} \xi_{324} \xi_{423} -
         \xi_{112} \xi_{214} \xi_{324} \xi_{423} - \xi_{123} \xi_{214} \xi_{312} \xi_{424} + \xi_{114} \xi_{223} \xi_{312} \xi_{424} +\\
          &  &\xi_{123} \xi_{212} \xi_{314} \xi_{424} - \xi_{112} \xi_{223} \xi_{314} \xi_{424} - \xi_{114} \xi_{212} \xi_{323} \xi_{424} +
          \xi_{112} \xi_{214} \xi_{323} \xi_{424}) + \\
 &\xi_{12}& (\xi_{124} \xi_{223} \xi_{314} \xi_{413} - \xi_{123} \xi_{224} \xi_{314} \xi_{413} -
         \xi_{124} \xi_{214} \xi_{323} \xi_{413} + \xi_{114} \xi_{224} \xi_{323} \xi_{413} + \\ & & \xi_{123} \xi_{214} \xi_{324} \xi_{413} -
         \xi_{114} \xi_{223} \xi_{324} \xi_{413} - \xi_{124} \xi_{223} \xi_{313} \xi_{414} + \xi_{123} \xi_{224} \xi_{313} \xi_{414} +\\
         &  &\xi_{124} \xi_{213} \xi_{323} \xi_{414} - \xi_{113} \xi_{224} \xi_{323} \xi_{414} - \xi_{123} \xi_{213} \xi_{324} \xi_{414} +
         \xi_{113} \xi_{223} \xi_{324} \xi_{414} + \\ & & \xi_{124} \xi_{214} \xi_{313} \xi_{423} - \xi_{114} \xi_{224} \xi_{313} \xi_{423} -
         \xi_{124} \xi_{213} \xi_{314} \xi_{423} + \xi_{113} \xi_{224} \xi_{314} \xi_{423} + \\ & & \xi_{114} \xi_{213} \xi_{324} \xi_{423} -
         \xi_{113} \xi_{214} \xi_{324} \xi_{423} - \xi_{123} \xi_{214} \xi_{313} \xi_{424} + \xi_{114} \xi_{223} \xi_{313} \xi_{424} +\\
         &  &\xi_{123} \xi_{213} \xi_{314} \xi_{424} - \xi_{113} \xi_{223} \xi_{314} \xi_{424} - \xi_{114} \xi_{213} \xi_{323} \xi_{424} +
         \xi_{113} \xi_{214} \xi_{323} \xi_{424}),\\
\end{array}
\end{equation*}

\begin{equation*}
\begin{array}{lrr}
C_2 = & -\xi_{24}& (\xi_{134} \xi_{223} \xi_{314} \xi_{413} - \xi_{123} \xi_{234} \xi_{314} \xi_{413} -
    \xi_{134} \xi_{214} \xi_{323} \xi_{413} + \xi_{114} \xi_{234} \xi_{323} \xi_{413} +  \\ & & \xi_{123} \xi_{214} \xi_{334} \xi_{413} -
    \xi_{114} \xi_{223} \xi_{334} \xi_{413} - \xi_{134} \xi_{223} \xi_{313} \xi_{414} + \xi_{123} \xi_{234} \xi_{313} \xi_{414} +\\
    & & \xi_{134} \xi_{213} \xi_{323} \xi_{414} - \xi_{113} \xi_{234} \xi_{323} \xi_{414} - \xi_{123} \xi_{213} \xi_{334} \xi_{414} +
    \xi_{113} \xi_{223} \xi_{334} \xi_{414} + \\ & & \xi_{134} \xi_{214} \xi_{313} \xi_{423} - \xi_{114} \xi_{234} \xi_{313} \xi_{423} -
    \xi_{134} \xi_{213} \xi_{314} \xi_{423} + \xi_{113} \xi_{234} \xi_{314} \xi_{423} + \\ & & \xi_{114} \xi_{213} \xi_{334} \xi_{423} -
    \xi_{113} \xi_{214} \xi_{334} \xi_{423} - \xi_{123} \xi_{214} \xi_{313} \xi_{434} + \xi_{114} \xi_{223} \xi_{313} \xi_{434} +\\
    & & \xi_{123} \xi_{213} \xi_{314} \xi_{434} - \xi_{113} \xi_{223} \xi_{314} \xi_{434} - \xi_{114} \xi_{213} \xi_{323} \xi_{434} +
    \xi_{113} \xi_{214} \xi_{323} \xi_{434}) +\\
 &\xi_{23}& (\xi_{134} \xi_{224} \xi_{314} \xi_{413} - \xi_{124} \xi_{234} \xi_{314} \xi_{413} -
    \xi_{134} \xi_{214} \xi_{324} \xi_{413} + \xi_{114} \xi_{234} \xi_{324} \xi_{413} + \\ & & \xi_{124} \xi_{214} \xi_{334} \xi_{413} -
    \xi_{114} \xi_{224} \xi_{334} \xi_{413} - \xi_{134} \xi_{224} \xi_{313} \xi_{414} + \xi_{124} \xi_{234} \xi_{313} \xi_{414} + \\
    & & \xi_{134} \xi_{213} \xi_{324} \xi_{414} - \xi_{113} \xi_{234} \xi_{324} \xi_{414} - \xi_{124} \xi_{213} \xi_{334} \xi_{414} +
    \xi_{113} \xi_{224} \xi_{334} \xi_{414} + \\ & & \xi_{134} \xi_{214} \xi_{313} \xi_{424} - \xi_{114} \xi_{234} \xi_{313} \xi_{424} -
    \xi_{134} \xi_{213} \xi_{314} \xi_{424} + \xi_{113} \xi_{234} \xi_{314} \xi_{424} + \\ & & \xi_{114} \xi_{213} \xi_{334} \xi_{424} -
    \xi_{113} \xi_{214} \xi_{334} \xi_{424} - \xi_{124} \xi_{214} \xi_{313} \xi_{434} + \xi_{114} \xi_{224} \xi_{313} \xi_{434} + \\
    & & \xi_{124} \xi_{213} \xi_{314} \xi_{434} - \xi_{113} \xi_{224} \xi_{314} \xi_{434} - \xi_{114} \xi_{213} \xi_{324} \xi_{434} +
    \xi_{113} \xi_{214} \xi_{324} \xi_{434})\\
 &-\xi_{14}& (\xi_{134} \xi_{224} \xi_{323} \xi_{413} - \xi_{124} \xi_{234} \xi_{323} \xi_{413} -
    \xi_{134} \xi_{223} \xi_{324} \xi_{413} + \xi_{123} \xi_{234} \xi_{324} \xi_{413} + \\ & & \xi_{124} \xi_{223} \xi_{334} \xi_{413} -
    \xi_{123} \xi_{224} \xi_{334} \xi_{413} - \xi_{134} \xi_{224} \xi_{313} \xi_{423} + \xi_{124} \xi_{234} \xi_{313} \xi_{423} +\\
    & & \xi_{134} \xi_{213} \xi_{324} \xi_{423} - \xi_{113} \xi_{234} \xi_{324} \xi_{423} - \xi_{124} \xi_{213} \xi_{334} \xi_{423} +
    \xi_{113} \xi_{224} \xi_{334} \xi_{423} + \\ & & \xi_{134} \xi_{223} \xi_{313} \xi_{424} - \xi_{123} \xi_{234} \xi_{313} \xi_{424} -
    \xi_{134} \xi_{213} \xi_{323} \xi_{424} + \xi_{113} \xi_{234} \xi_{323} \xi_{424} + \\ & & \xi_{123} \xi_{213} \xi_{334} \xi_{424} -
    \xi_{113} \xi_{223} \xi_{334} \xi_{424} - \xi_{124} \xi_{223} \xi_{313} \xi_{434} + \xi_{123} \xi_{224} \xi_{313} \xi_{434} +\\
    & & \xi_{124} \xi_{213} \xi_{323} \xi_{434} - \xi_{113} \xi_{224} \xi_{323} \xi_{434} - \xi_{123} \xi_{213} \xi_{324} \xi_{434} +
    \xi_{113} \xi_{223} \xi_{324} \xi_{434}) +\\
 &\xi_{13}& (\xi_{134} \xi_{224} \xi_{323} \xi_{414} - \xi_{124} \xi_{234} \xi_{323} \xi_{414} -
    \xi_{134} \xi_{223} \xi_{324} \xi_{414} + \xi_{123} \xi_{234} \xi_{324} \xi_{414} + \\ & & \xi_{124} \xi_{223} \xi_{334} \xi_{414} -
    \xi_{123} \xi_{224} \xi_{334} \xi_{414} - \xi_{134} \xi_{224} \xi_{314} \xi_{423} + \xi_{124} \xi_{234} \xi_{314} \xi_{423} +\\
    & & \xi_{134} \xi_{214} \xi_{324} \xi_{423} - \xi_{114} \xi_{234} \xi_{324} \xi_{423} - \xi_{124} \xi_{214} \xi_{334} \xi_{423} +
    \xi_{114} \xi_{224} \xi_{334} \xi_{423} + \\ & & \xi_{134} \xi_{223} \xi_{314} \xi_{424} - \xi_{123} \xi_{234} \xi_{314} \xi_{424} -
    \xi_{134} \xi_{214} \xi_{323} \xi_{424} + \xi_{114} \xi_{234} \xi_{323} \xi_{424} + \\ & & \xi_{123} \xi_{214} \xi_{334} \xi_{424} -
    \xi_{114} \xi_{223} \xi_{334} \xi_{424} - \xi_{124} \xi_{223} \xi_{314} \xi_{434} + \xi_{123} \xi_{224} \xi_{314} \xi_{434} +\\
    & & \xi_{124} \xi_{214} \xi_{323} \xi_{434} - \xi_{114} \xi_{224} \xi_{323} \xi_{434} - \xi_{123} \xi_{214} \xi_{324} \xi_{434} +
    \xi_{114} \xi_{223} \xi_{324} \xi_{434}) +\\
 &\xi_{34}& (\xi_{134} \xi_{224} \xi_{323} \xi_{414} - \xi_{124} \xi_{234} \xi_{323} \xi_{414} -
    \xi_{134} \xi_{223} \xi_{324} \xi_{414} + \xi_{123} \xi_{234} \xi_{324} \xi_{414} + \\ & & \xi_{124} \xi_{223} \xi_{334} \xi_{414} -
    \xi_{123} \xi_{224} \xi_{334} \xi_{414} - \xi_{134} \xi_{224} \xi_{314} \xi_{423} + \xi_{124} \xi_{234} \xi_{314} \xi_{423} +\\
    & & \xi_{134} \xi_{214} \xi_{324} \xi_{423} - \xi_{114} \xi_{234} \xi_{324} \xi_{423} - \xi_{124} \xi_{214} \xi_{334} \xi_{423} +
    \xi_{114} \xi_{224} \xi_{334} \xi_{423} + \\ & & \xi_{134} \xi_{223} \xi_{314} \xi_{424} - \xi_{123} \xi_{234} \xi_{314} \xi_{424} -
    \xi_{134} \xi_{214} \xi_{323} \xi_{424} + \xi_{114} \xi_{234} \xi_{323} \xi_{424} + \\ & & \xi_{123} \xi_{214} \xi_{334} \xi_{424} -
    \xi_{114} \xi_{223} \xi_{334} \xi_{424} - \xi_{124} \xi_{223} \xi_{314} \xi_{434} + \xi_{123} \xi_{224} \xi_{314} \xi_{434} +\\
    & & \xi_{124} \xi_{214} \xi_{323} \xi_{434} - \xi_{114} \xi_{224} \xi_{323} \xi_{434} - \xi_{123} \xi_{214} \xi_{324} \xi_{434} +
    \xi_{114} \xi_{223} \xi_{324} \xi_{434}).\\
\end{array}
\end{equation*}
\end{small}
\end{example}

\end{document}